\documentclass[reqno]{amsart}
 
\usepackage{amsmath}
\usepackage{amsthm,amssymb,color,comment}
\usepackage{bbm}
\usepackage{hyperref}
\usepackage[TS1,T1]{fontenc}
\usepackage[utf8]{inputenc}
\usepackage{dsfont}
\usepackage{tikz}
\usepackage{dsfont}
\usetikzlibrary{arrows.meta,calc,positioning,shadows,shapes}
\usepackage{enumitem}
\usepackage{textcomp}
\usepackage{subfigure}
\usepackage{accents}
\usepackage{bm}
\usepackage{mathtools}
\usepackage{bbold}
\usepackage[sort]{cite}% to arrange citations in increasing order

\numberwithin{equation}{section}

\newtheorem{thm}{Theorem}[section]
\newtheorem{proposition}[thm]{Proposition}
\newtheorem{lem}[thm]{Lemma}

\newtheorem{Def}[thm]{Definition}

\theoremstyle{definition}
\newtheorem{Ass}[thm]{Assumption}
\newtheorem{rem}[thm]{Remark}

\newtheorem{Ex}[thm]{Example}

\DeclareMathOperator*{\esssup}{ess\,sup}
\DeclareMathOperator*{\essinf}{ess\,inf}
\DeclareMathOperator{\DIV}{div}
\DeclareMathOperator{\Span}{span}

\DeclareMathOperator{\sgn}{sgn}

\newcommand{\Rho}{\mathrm{P}}
\newcommand{\Bu}{\textbf{U}}
\newcommand{\bwi}{\bm{\omega_i}}
\newcommand{\bwj}{\bm{\omega_j}}

\newcommand{\R}{\mathbb{R}}
\newcommand{\Td}{\mathbb{T}^{d}}

\newcommand{\p}{\partial}
\newcommand {\f}{\frac}
\newcommand{\eps}{\varepsilon}

\newcommand{\diff}{\mathop{}\!\mathrm{d}}
\newcommand{\ie}{\emph{i.e.}\;}
\newcommand{\bu}{\textbf{u}}

\newcommand{\weaks}{\overset{\ast}{\rightharpoonup}}

\DeclarePairedDelimiter{\norm}{\lVert}{\rVert}

\tikzstyle{line} = [draw]

\tikzset{label/.style={draw=gray, ultra thin, rounded corners=.25ex, fill=gray!20,text width=4cm, text badly centered,  inner sep=2ex, anchor=east, minimum height=4em}}

\makeatletter
\newcommand{\doublewidetilde}[1]{{%
  \mathpalette\double@widetilde{#1}%
}}
\newcommand{\double@widetilde}[2]{%
  \sbox\z@{$\m@th#1\widetilde{#2}$}%
  \ht\z@=.9\ht\z@
  \widetilde{\box\z@}%
}
\makeatother
\author{Charles Elbar}
\address{{\it Charles Elbar:} Sorbonne Universit\'{e}, Laboratoire Jacques-Louis Lions (LJLL), F-75005 Paris, France}
\email{charles.elbar@sorbonne-universite.fr}
\thanks{}

\author{Piotr Gwiazda}
\address{{\it Piotr Gwiazda:} Institute of Mathematics of Polish Academy of Sciences}
\email{pgwiazda@mimuw.edu.pl}
\thanks{P.G. was supported by the National Science Center (Poland), project 2018/31/B/ST1/02289.}

\author{Jakub Skrzeczkowski}
\address{{\it Jakub Skrzeczkowski: } Faculty of Mathematics, Informatics and Mechanics, University of Warsaw, Poland}
\email{jakub.skrzeczkowski@student.uw.edu.pl}
\thanks{J.S. and A.Ś.-G. were supported by the National Science Center (Poland), project 2018/30/M/ST1/00423.}

\author{Agnieszka Świerczewska-Gwiazda}
\address{{\it Agnieszka Świerczewska-Gwiazda:} Faculty of Mathematics, Informatics and Mechanics, University of Warsaw, Poland}
\email{aswiercz@mimuw.edu.pl}
\thanks{}

\begin{document}

\title[]{From nonlocal Euler-Korteweg to local Cahn-Hilliard via the high-friction limit}

\begin{abstract}
Several recent papers considered the high-friction limit for systems arising in fluid mechanics. Following this approach, we rigorously derive the nonlocal Cahn-Hilliard equation as a limit of the nonlocal Euler-Korteweg equation using the relative entropy method. 
%, we prove rigorously that the solutions converge to the nonlocal Cahn-Hilliard equation. 
Applying the recent result by the first and third author, we also derive rigorously {the large-friction nonlocal-to-local limit}. %{Consequently, by choosing the nonlocal primitive system and following two-step procedure we are able to derive the equation without making ...} 
The proof is formulated for dissipative measure-valued solutions of the nonlocal Euler-Korteweg equation which are known to exist on arbitrary intervals of time. 
%As observed by Lattanzio and Tzavaras ({\it Comm. PDEs}, 2017), for the Euler-Korteweg equation, the limit yields the degenerate Cahn-Hilliard equation. However, their proof assumes certain properties of the limit system which are not guaranteed by the current theory. To overcome this problem, we propose to consider the nonlocal Euler-Korteweg equation.
Our work provides a new method to derive equations not enjoying classical solutions via relative entropy method by introducing the nonlocal effect in the fluid equation. 
\end{abstract}



\maketitle
\setcounter{tocdepth}{1}

\section{Introduction}

We consider the nonlocal Euler-Korteweg system re-scaled in time \ie $t\to \f{t}{\eps}$ and with high friction coefficient $\f{1}{\eps}$
\begin{align}
&\p_{t}\rho+\f{1}{\eps}\DIV(\rho \bu)=0,\quad &\text{in}\quad &(0,+\infty)\times \Td,\label{eq:EK1}\\
&\p_{t}(\rho \bu)+\f{1}{\eps}\DIV\left(\rho\bu\otimes\bu\right)=-\f{1}{\eps^{2}} \rho\bu-\f{1}{\eps}\rho\nabla(F'(\rho)+B_{\eta}[\rho]),\quad &\text{in}\quad &(0,+\infty)\times \Td.\label{eq:EK2}
\end{align}
This equation models the long-time asymptotics of the motion of a compressible fluid with density $\rho$, velocity $\bu$ which is in fact a liquid-vapor mixture. The fluid experiences high friction (due to the term $-\f{1}{\eps^{2}} \rho\bu$) and additional capillary effects in the transition zone between liquid and vapor (due to the term $-\f{1}{\eps}\rho\nabla(F'(\rho)+B_{\eta}[\rho])$ as proposed by Korteweg \cite{kortewlg1901forme}).\\

Concerning the notation, $\Td$ is the $d$-dimensional flat torus, $\eps>0$, $B_{\eta}$ is the nonlocal operator approximating $-\Delta$ operator, defined by
\begin{equation}\label{operatorB}
B_\eta[\rho](x) = \f{1}{\eta^{2}}(\rho(x)-\omega_{\eta}\ast \rho(x))=\f{1}{\eta^{2}}\int_{\Td}\omega_{\eta}(y)(\rho(x)-\rho(x-y)) \diff y
\end{equation}
for $\eta>0$ small enough and $\omega_{\eta}$ is the usual radial mollification kernel $\omega_{\eta}(x)=\frac{1}{\eta^{d}}\omega(\frac{x}{\eta})$ with $\omega$ compactly supported in the unit ball of $\R^{d}$ satisfying
\begin{equation}
\int_{\R^{d}}\omega(y) \diff y =1, \quad \int_{\R^{d}} y\, \omega(y) \diff y =0,  \quad \int_{\R^{d}}   y_i y_j \omega \diff y = \delta_{i,j}\f{2D}{d}  < \infty.
\label{as:omega}
\end{equation}

When $\eps$ is very small, the friction is so big, that we mostly observe a phase separation phenomenon between the liquid and the vapor. More rigorously, when $\eps \to 0$, we prove that the constructed solution of~\eqref{eq:EK1}-\eqref{eq:EK2} converge to solutions of the nonlocal Cahn-Hilliard  \begin{align}
\partial_t \rho = \DIV(\rho \nabla \mu),\quad \text{in}\quad &(0,+\infty)\times \Td,\label{eq:CHNL1}\\
\mu = B_{\eta}[\rho] + F'(\rho),\quad \text{in}\quad &(0,+\infty)\times \Td \label{eq:CHNL2},
\end{align}
see Theorem \ref{thm:conv_EK_CH_nonloc}. Furthermore, when $\eps, \eta \to 0$ in some scaling to be determined, we prove the convergence of~\eqref{eq:EK1}-\eqref{eq:EK2} to the local Cahn-Hilliard equation
 \begin{align}
\partial_t \rho = \DIV(\rho \nabla \mu),\quad \text{in}\quad &(0,+\infty)\times \Td,\label{eq:CH1}\\
\mu = -D\Delta \rho + F'(\rho),\quad \text{in}\quad &(0,+\infty)\times \Td, \label{eq:CH2}
\end{align}
which describes the dynamics of phase separation, see Theorem \ref{thm:conv_EK_CH_loc}. \\

Our proof relies on the relative entropy method, which is for instance often used in the context of weak-strong uniqueness. It relies on certain regularity of solutions of the limit system, which is not available in the case of the local Cahn-Hilliard equation. Therefore, we introduce an intermediate step, which is interesting by itself, and consider the nonlocal Cahn-Hilliard equation by introducing the parameter $\eta$. Since we know from~\cite{MR4574535} that the solutions to the nonlocal Cahn-Hilliard equation converge to the weak solutions of the local Cahn-Hilliard equation (see Definition \ref{def:weak_sol_local}) when $\eta\to 0$, it remains to prove that the nonlocal Euler-Korteweg system tends to the nonlocal Cahn-Hilliard equation when $\eps\to 0$. Then, sending $\eps$ and $\eta$ to 0 with the appropriate scaling, we prove the result.    

\begin{figure}[!htb]
\centering
\resizebox{\textwidth}{!}{%
\begin{tikzpicture}[node distance=7.8cm,auto]
%,>=latex'
% Place nodes
\node [label] (init) {Nonlocal Euler-Korteweg \eqref{eq:EK1}-\eqref{eq:EK2}};
\node [label, right of=init] (node) {Non-local degenerate Cahn-Hilliard \eqref{eq:CHNL1}-\eqref{eq:CHNL2}};
\node [label, right of=node] (node1) {Local degenerate Cahn-Hilliard \eqref{eq:CH1}-\eqref{eq:CH2}};

% Draw edges
\draw[-{Latex[length=3mm, width=3mm]}] (init) -- node[below] {Theorem~\ref{thm:conv_EK_CH_nonloc}} node[above] {$\eps\to 0$} (node);

\draw[-{Latex[length=3mm, width=3mm]}] (node) -- node[below] {Proved in \cite{MR4574535}} node[above] {$\eta\to 0$} (node1);

\draw[-{Latex[length=3mm, width=3mm]}] (init) to[bend right] node[midway,below,inner sep=2pt] {Theorem~\ref{thm:conv_EK_CH_loc} } node[midway,above,inner sep=8pt] {$\eps\to 0$, $\eta\to 0$ together} (node1);

%node[above] {open problem} node[below] {$\omega^S, \omega_{\alpha}^L \weaks \delta_0$}

\end{tikzpicture}
}
\caption{Relation between the three equations considered in this article.}
\label{fig:results}
\end {figure}
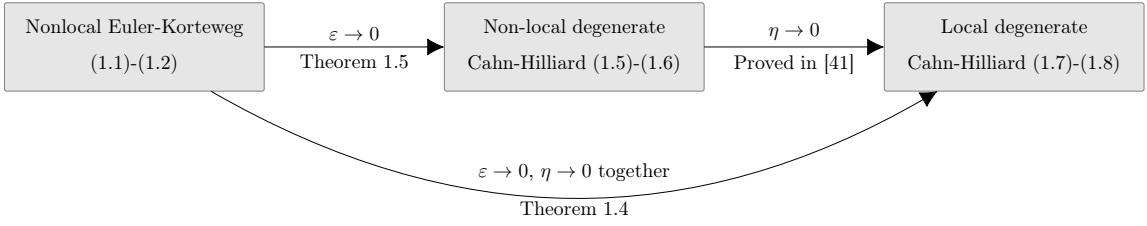

The main motivation for our work is the paper of Lattanzio and Tzavaras \cite{MR3615546}, who prove the convergence of the local Euler Korteweg system to the local Cahn-Hilliard equation. They assume the existence of admissible weak solutions of the first system and classical solutions of the second one. The first assumption is a drawback as dissipative (that is, satisfying energy inequality) weak solutions existing on arbitrary intervals of time are not known to exist for most models in fluids dynamics. One can try to construct the solutions via the convex integration method but these solutions will have a jump in the energy at the initial time so they will not be dissipative. The second assumption of \cite{MR3615546} is also difficult to be satisfied as so far, there is no theory of classical solutions to the local Cahn-Hilliard equation with degenerate mobility on arbitrary intervals of time. Similarly, there is no maximum principle that is necessary in \cite{MR3615546} to deduce that the classical solution is strictly positive using positivity of the initial condition.\\

We propose to overcome the first problem by the concept of dissipative measure-valued solutions, introduced by DiPerna \cite{MR775191} in the context of hyperbolic conservation laws in one dimension and by DiPerna and Majda \cite{MR877643} for the incompressible Euler equations. Roughly speaking, they are defined as the weak limit of classical solutions of appropriate approximating problems. As weak compactness is not sufficient to pass to the limit in nonlinear terms, the definition of the measure-valued solution includes the Young measure $\nu_{t,x}$ and the concentration measure $m$ to represent weak limits as in \eqref{eq:notation_bar}.\\

While measure-valued solutions are weaker than the usual weak solutions, they are dissipative and they are known to exist. Moreover, their importance comes from the fact that they enjoy the property called weak-strong uniqueness: they coincide with the strong solution whenever the latter exists. The dissipativity is important both for the weak-strong uniqueness and application of the relative entropy method: the weak-strong uniqueness does not hold for weak or measure-valued solutions without any condition on energy as demonstrated by solutions arising by the convex integration method \cite{MR3340997, MR2600877}.\\

Since the weak-strong uniqueness property was observed by Brenier, De Lellis and Székelyhidi in \cite{MR2805464}, measure-valued solutions were studied for several systems including compressible fluid models \cite{MR3424896}, isentropic Euler system \cite{MR4154934},  polyconvex elastodynamics \cite{MR2960036}, Euler-Poisson system \cite{Debiec}, general hyperbolic conservation laws \cite{MR4093617}. Moreover, for many equations describing compressible fluids, the measure-valued formulation has been significantly simplified \cite{MR4271964,MR4416230,MR4102807}: it boils down to the usual distributional identity modulo the so-called Reynolds stress tensor.    \\

Concerning the problem of the existence of classical solutions, we propose to introduce nonlocality in the equation and introduce an intermediate step in the convergence analysis as outlined in Figure~\ref{fig:results}. The advantage is that the nonlocal Cahn-Hilliard equation is in fact a porous medium equation. In particular, it satisfies the maximum principle and so, if the initial condition is positive, the solution remains positive and one can prove the existence and uniqueness of a classical solution, see Section~\ref{sect:class_sol}. Furthermore, it is known that the nonlocal Cahn-Hilliard equation converges to the local one \cite{MR4574535} so that at the end, the nonlocality can be removed.   \\

To prove the convergence, we use the relative entropy method. The method is based on introducing a functional called relative entropy (or energy), which measures the dissipation between two solutions of the system. Essentially, the same method is used to prove the aforementioned weak-strong uniqueness when the relative entropy measures the distance between weak (measure-valued) and strong solutions. This strategy has been applied for several singular limits \cite{MR4149680,MR4126770,MR4349783,MR4160173,MR3987806,MR3615546,MR3056757} and we also refer to the excellent review on weak-strong uniqueness~\cite{MR3838055}.\\

Our proof via the relative entropy method is based on an important assumption that the initial datum is well-prepared. In our case, this means that the initial velocity ${ \bu^0}$ vanishes as the parameter $\varepsilon \to 0$ cf. \eqref{eq:thm14_velocity_to_zero} and \eqref{eq:thm14_velocity_to_zero_2}  so that the initial kinetic energy is very small. Such an assumption is necessary to guarantee that the relative entropy $\Theta(0)$ at time $t=0$ converges to 0 as $\varepsilon \to 0$ so that $\Theta(t) \to 0$, cf. \eqref{est:relative_entropy_4}, which implies the main result. Let us however remark that one can also study similar problems via compactness methods and this approach is also effective for ill-prepared initial data. Nevertheless, its applicability is restricted to some special cases like one spatial dimension (which allows to use the div-curl lemma in the time-space setting) \cite{MR1042662} or the presence of viscosity terms yielding compactness \cite{MR2384572}. 

\subsection{Rigorous formulation of the main result}
We make the following assumptions on the potential $F$.

\begin{Ass}[potential $F$]\label{ass:potentialF} For the interaction potential we assume that there exists $k \geq 2$ and constant $C$ such that $F$ can be written as $F=F_{1}+F_{2}$ where
\begin{enumerate}
    \item  $F_{1}\in C^{4}(\R)$ is a convex, nonnegative function having $k$-growth
    $$
    \frac{1}{C} |u|^k - C  \leq F_1(u) \leq C|u|^k + C,
    $$
    $$
    \frac{1}{C} |u|^{(k-2)} - C\leq F_1''(u) \leq C|u|^{(k-2)} + C
    $$
    and satisfying $|u F_{1}'(u)|\le C (F_{1}(u)+1)$, $|u F_{1}^{(3)}(u)|\le C (F_{1}''(u)+1)$, 
    \item $F_{2}\in C^{4}(\R)$ is such that $F_{2}, F_2', F_{2}'', sF_{2}^{(3)}(s) \in L^{\infty}(\R)$ are bounded on the whole line. Moreover, $\|F_2''\|_{\infty} < \frac{1}{C_P}$ where $C_P$ is a constant in Lemma \ref{lem:Poincaré_nonlocal_H1_L2}.
\end{enumerate}
We also define $s := \frac{2k}{k-1}$ and $s'$ its conjugate exponent.
\end{Ass}

\begin{Ex}\label{ex:potentials}
The following potentials satisfy Assumption \ref{ass:potentialF}.
\begin{enumerate}[label=(\arabic*)]    \item power-type potential $F(u) = |u|^{\gamma}$, $\gamma>2$ used in the context of tumor growth models \cite{Perthame-Hele-Shaw,david2021incompressible,MR4504016,DEBIEC2021204},
    \item double-well potential $F(u) = u^2\,(u-1)^2$ which is an approximation of logarithmic double-well potential often used in Cahn-Hilliard equation, see \cite[Chapter 1]{MR4001523}. 
\end{enumerate}

\end{Ex}

Before stating the main result, we define solutions of the local degenerate Cahn-Hilliard equation.
\begin{Def}
\label{def:weak_sol_local}\noindent We say that $\rho$ is a weak solution of \eqref{eq:CH1}-\eqref{eq:CH2} if
\begin{align*}
&\rho \in L^{\infty}(0,T;L^{k}(\Td))\cap L^{2}(0,T;H^{2}(\Td))\cap {L^{\infty}(0,T;H^{1}(\Td))},  \quad \p_{t} \rho \in L^{2}(0,T;W^{-1,s'}(\Td)),  \\
&\sqrt{F_{1}''(\rho)}\nabla \rho\in L^2((0,T)\times\Td),
\end{align*}
 $\rho(0,x)={\rho^0}(x)$ a.e. in $\Td$ and if for all $\varphi \in L^2(0,T; W^{2,\infty}(\Td))$ we have 
\begin{multline*}
\int_{0}^{T}\langle\p_{t}\rho,\varphi\rangle_{(W^{-1,s'}(\Td),W^{1,s}(\Td))} = -D\int_0^T \int_{\Td}\Delta \rho\, \nabla \rho \cdot \nabla\varphi 
 { \, \diff x \diff t} -D\int_0^T \int_{\Td}\rho\,\Delta \rho\,\Delta\varphi { \, \diff x \diff t} \\-\int_{0}^{T}\int_{\Td}\rho\,F''(\rho)\,\nabla \rho\cdot\nabla\varphi { \, \diff x \diff t}.     
\end{multline*}
\end{Def}

{Note that the term $\int_0^T \int_{\Td}\Delta \rho\,\nabla \rho \cdot \nabla \varphi \, \diff x \diff t$ is well-defined because, due to the assumptions on $\rho$, $\Delta \rho\,\nabla \rho \in L^{\infty}(0,T; L^1(\Td))$.}\\

The definition of dissipative measure-valued solutions to \eqref{eq:EK1}--\eqref{eq:EK2} is quite technical and will be presented in Definitions \ref{def:dissipative_measure} and \ref{def:diss_mvs}. The main theorem reads as follows. 

\begin{thm}\label{thm:conv_EK_CH_loc}
Let ${\rho^0}$ be an initial density satisfying 
$$
{\rho^0}\geq \sigma > 0, \qquad {\rho^0} \in C^{3}(\Td)
$$ 
for some $\sigma>0$. Let ${ \bu^0_{\varepsilon}}$ be an initial velocity satisfying
\begin{equation}\label{eq:thm14_velocity_to_zero}
\|{ \bu^0_{\varepsilon}}\|_{L^2(\Td)} \to 0 \mbox{ as } \varepsilon \to 0.
\end{equation}
Let $(\rho_{\eta, \varepsilon}, \overline{\sqrt{\rho_{\eta, \varepsilon}}\bu_{\eta, \varepsilon}}, \nu^{\eta, \varepsilon}, m_{\eta, \varepsilon})$ be a dissipative measure-valued solution of \eqref{eq:EK1}--\eqref{eq:EK2} with the initial condition $({\rho^0},{ \bu^0_{\varepsilon}})$ and parameters $\varepsilon, \eta$ satisfying Poincaré inequality \eqref{eq:nonloc_Poinc_stronger_bar}. Then, for each sequence $\eta_k \to 0$, there exists a subsequence $\{\eta_k\}$ (not relabelled) and a sequence $\{\varepsilon_k\}$ depending on $\eta_{k}$ and the final time $T$ such that $\varepsilon_k \to 0$ and $\rho_{\eta_k, \varepsilon_k} \to \rho$ in $L^{2}(0,T; L^2(\Td))$, where $\rho$ is a weak solution of \eqref{eq:CH1}--\eqref{eq:CH2} with initial condition ${\rho^0}$ as defined in Definition \ref{def:weak_sol_local}.
\end{thm}

Let us briefly comment that the measure-valued solution has in fact four components. While the first component $\rho_{\eta,\eps}$ is the most important since it converges to the Cahn-Hilliard equation, we can also characterize what happens with the other ones, see Theorem \ref{thm:convergence_YM_eps_and_eta}. Roughly speaking, the second one converges to 0 in $L^{\infty}(0,T; L^2(\Td))$ which represents that in the high-friction limit, the kinetic energy converges to 0. The parametrized measure $\nu^{\eta, \eps}$ converges in the second Wasserstein metric $\mathcal{W}_2$ to the Dirac mass $\delta_{\rho(t,x)} \otimes \delta_{\bm{0}}$:
$$
\int_0^T \int_{\Td} \left[\mathcal{W}_2(\nu^{\eta_k, \eps_k}, \delta_{\rho(t,x)}\otimes \delta_{\bm{0}})\right]^2\diff x \diff t \to 0 \mbox{ as } \varepsilon_k, \eta_k \to 0
$$
while the concentration measure $m_{\eta_k, \eps_k }$ converges to 0 in the total variation norm. The estimate in the Wasserstein metric is in the spirit of \cite{MR3509212}.\\

Theorem \ref{thm:conv_EK_CH_loc} is valid only for a subsequence as the convergence from non-local Cahn-Hilliard to the local one is based on the compactness arguments (and there is no uniqueness for the limit equation). On the other hand, the passage from the nonlocal Euler-Korteweg equation to the nonlocal Cahn-Hilliard equation is based on the relative entropy method and so the convergence is satisfied for any sequence. We state this result below.

\begin{thm}\label{thm:conv_EK_CH_nonloc}
Let $\eta \in (0, \eta_0)$ where $\eta_0$ is defined in Lemma \ref{lem:Poincaré_nonlocal_H1_L2}. Let ${\rho^0}$ be an initial density satisfying 
$$
{\rho^0}\geq \sigma > 0, \qquad {\rho^0} \in C^{3}(\Td)
$$ 
for some $\sigma>0$. Let ${ \bu^0_{\varepsilon}}$ be an initial velocity satisfying
\begin{equation}\label{eq:thm14_velocity_to_zero_2}
\|{ \bu^0_{\varepsilon}}\|_{L^2(\Td)} \to 0 \mbox{ as } \varepsilon \to 0.
\end{equation}
Let $(\rho_{\eta, \varepsilon}, \overline{\sqrt{\rho_{\eta, \varepsilon}}\bu_{\eta, \varepsilon}}, \nu^{\eta, \varepsilon}, m_{\eta, \varepsilon})$ be a dissipative measure-valued solution of \eqref{eq:EK1}--\eqref{eq:EK2} with initial condition $({\rho^0},{ \bu^0_{\varepsilon}})$ and parameters $\varepsilon, \eta$ satisfying Poincaré inequality \eqref{eq:nonloc_Poinc_stronger_bar}. Let $\rho_{\eta}$ be the solution of non-local Cahn-Hilliard \eqref{eq:CHNL1}-\eqref{eq:CHNL2} with the same initial condition ${\rho^0}$. Then, $\rho_{\eta, \varepsilon} \to \rho_{\eta}$ in $L^{\infty}(0,T;L^2(\Td))$ as $\eps\to0$.
\end{thm}
Similarly as for Theorem \ref{thm:conv_EK_CH_loc}, we can prove convergence of the other components of the measure-valued solution $\overline{\sqrt{\rho_{\eta, \varepsilon}}\bu_{\eta, \varepsilon}}$, $\nu^{\eta, \varepsilon}$, $m_{\eta, \varepsilon}$, see Theorem \ref{thm:convergence_YM_only_eps}.

\section{Physical relevancy of the system}

\paragraph{\underline{The Euler--Korteweg equation}}
The compressible Euler--Korteweg equation models the motion of liquid-vapor mixtures with possible phase transitions. It combines the classical Euler equation with Korteweg tensor introduced in \cite{kortewlg1901forme}. The equation reads
\begin{equation}\label{eq:E-K-equaton_general}
\begin{split}
&\p_{t}\rho+\DIV(\rho\bu)=0,\\
&\p_{t}(\rho\bu)+\DIV(\rho\bu\otimes\bu)+\nabla(p(\rho))=-\zeta\rho\bu+\rho\nabla(K(\rho)\Delta\rho+\f{1}{2}K'(\rho)|\nabla\rho|^{2}).
\end{split}
\end{equation}
Here, $\rho$ is the density of the fluid, $\bu$ is its velocity, $K(\rho)$ corresponds to the capillary coefficient, $\zeta$ is the friction coefficient and $p$ is the pressure function. In a liquid-vapor system, the tensor $K$ takes into account that the liquid and vapour are separated by a thin layer of finite thickness and describes the capillary effects in this transition zone. There are numerous mathematical results concerning well-(and ill-)posedness of solutions to \eqref{eq:E-K-equaton_general}, see \cite{MR3341206, MR3613503,MR3961293,MR2481754,MR2354691,MR1978317}. For instance, for some particular choice of $K(\rho)$, an approach to prove existence of global solutions is to relate the Euler-Korteweg and the Schrödinger equation through the Madelung transform~\cite{MR2979973,MR2481754}. For more general cases, only local existence~\cite{MR2354691} and global existence for small irrotational data~\cite{MR3613503} is known. For the physical background of \eqref{eq:E-K-equaton_general} (in particular, the form of the Korteweg tensor) we refer to \cite{MR775366,MR1836527,MR2760987} but it is a fairly complicated matter.  \\

The viscous version of \eqref{eq:E-K-equaton_general}, that is the Navier-Stokes-Korteweg system, was also studied in the mathematical literature \cite{MR4412067,MR3433629}. In particular, several papers are concerned with the case of the nonlocal equation, where $-\Delta \rho$ is approximated by the nonlocal operator $B_{\eta}$. In~\cite{MR2184845}, the author proves the short-time well-posedness while in~\cite{MR3008260}, the global well-posedness as well as the convergence of the nonlocal Navier-Stokes-Korteweg to the local one is established. We also refer to~\cite{MR3583169} for a variant of this system. \\

\paragraph{\underline{The high-friction limit}} The high-friction limit (also referred to in the literature as the relaxation limit) is a part of a long research program of establishing a connection between nonlinear hyperbolic systems and degenerate diffusion equations. One of the first results in this direction \cite{MR1042662} states that the solutions to the compressible Euler equations in one dimension
\begin{equation}\label{eq:1d_euler}
\begin{split}
&\p_{t} \rho + \p_{x}(\rho\, u) = 0,\\
& \varepsilon^2 \p_t (\rho \, u ) + \p_x(\varepsilon^2 \, \rho \, u^2 + p(\rho) ) = - u
\end{split}
\end{equation}
converge, as $\varepsilon \to 0$, to the porous media equation
$$
\partial_t \rho =  \partial_x\left({\rho}\, \partial_x p(\rho) \right)
$$
where $p(\rho)$ is the pressure function of the form $p(\rho) = \rho^{\gamma}$. To connect \eqref{eq:1d_euler} with our system \eqref{eq:EK1}--\eqref{eq:EK2}, it is sufficient to rescale $\widetilde{u} = \varepsilon \, u$ so that we have
\begin{equation}\label{eq:1D_euler_rescaled}
\begin{split}
&\p_{t} \rho + \frac{1}{\varepsilon} \p_{x}(\rho\, \widetilde{u}) = 0,\\
& \p_t (\rho \, \widetilde{u}) + \frac{1}{\varepsilon}\, \p_x( \, \rho \, {\widetilde{u}}^2 + p(\rho) ) = - \frac{\widetilde{u}}{\varepsilon^2}.
\end{split}
\end{equation}
Intuitively, it is easy to understand from \eqref{eq:1D_euler_rescaled} that the flow of the fluid with big damping or friction (caused by the term $-\frac{\widetilde{u}}{\varepsilon^2}$) and very small kinetic energy (caused by the initial condition) resembles a flow through a porous media. Several other limit passages have been studied between porous medium equation and hyperbolic equations \cite{MR870025, MR917603, MR898047,lemarie2023parabolic,MR4612415,MR4548999}. The revival of interest in this type of problem appeared recently with an observation that one can study these problems by the relative entropy method \cite{MR3056757, MR3594360, MR4149680, MR3615546,gallenmuller2023cahn}.  \\  

In our case, we consider \eqref{eq:E-K-equaton_general} with $K(\rho)=1$, large friction coefficient $\zeta=\f{1}{\eps}$, we approximate the Laplace operator $-\Delta$ by the nonlocal operator $B_{\eta}$ with $\eta$ small enough, and we perform a rescaling in time $t\to\f{t}{\eps}$. Then, we let both $\eps, \eta \to 0$, and in the limit, we obtain the Cahn-Hilliard equation. Again, it is intuitive that due to the very large damping and small kinetic energy, we observe mostly a phase separation process. The latter is described by the Cahn-Hilliard equation so it is not surprising that it is the limiting PDE. \\

\paragraph{\underline{The Cahn-Hilliard equation}}

In their publications \cite{CAHN1961795} and \cite{Cahn-Hilliard-1958} J.W. Cahn and J. E. Hilliard proposed the equation in 1958. It represents now a commonly used mathematical model for describing phase changes in fluids, although the equation was primarily developed in material sciences to explore phase separation processes under isotropy and constant temperature circumstances.\\

Being of fourth-order, the (local) Cahn-Hilliard equation is often rewritten in a system of two second-order equations, \ie
\begin{equation}
    \p_t \rho = \DIV\left(m(\rho)\nabla \left(F^\prime(\rho)-D \Delta \rho \right)\right) \to \begin{cases}
        \p_t \rho &= \DIV\left(m(\rho)\nabla \mu \right),\\
        \mu &= -D \Delta \rho +F^\prime(\rho),
    \end{cases}
    \label{eq:Cahn-Hilliard}
\end{equation}
where $\rho$ is the concentration of a phase and $\mu$ is called the chemical potential in material sciences but is often used as an effective pressure. The interaction potential $F(\rho)$ contained in this effective pressure term comprises the effects of attraction and repulsion between particles. Finally, the Laplace operator takes into account surface tension effects.\\ 

The existence and uniqueness of solutions for the Cahn-Hilliard system \eqref{eq:Cahn-Hilliard} strictly depends on the properties of the mobility term $m(\rho)$ and the potential $F(\rho)$, as well as the conditions assigned on the boundary. More specifically, the presence of \textit{degeneration} on the mobility, i.e. the possibility for it to vanish, can turn the analysis of solutions into a rather complex problem. We refer to~\cite{MR1377481} for the first existence result of weak solutions in the case of degenerate mobility and to~\cite{MR3448925} for some improvements. The uniqueness and the existence of classical solutions are open questions for this type of mobility. Since we use a relative entropy argument between the Euler-Korteweg equation and the Cahn-Hilliard equation, the existence of a classical solution of the latter is a crucial point. In fact, we need an $L^{\infty}$ bound on the second derivative of $\rho$, and that the solution remains positive for all times in a finite time interval. Since the Cahn-Hilliard equation does not satisfy the maximum principle, it is impossible to get these estimates. For that purpose, we introduce the nonlocal Cahn-Hilliard equation, which is a second-order equation with a nonlocal smooth advection term. With classical arguments, we are able to prove that the latter admits a unique positive classical solution. Let us finally remark that the nonlocal Cahn-Hilliard equation is nowadays a topic of intense research activity, see for instance \cite{Poiatti2022,MR4198717,MR4241616,MR3688414,carrillo2023degenerate}.\\

\paragraph{\underline{Mobilities}}
This work focuses on the mobility case where $m(\rho)=\rho$, which is a result of deriving the Cahn-Hilliard equation from fluid models. This mobility is also obtained from Vlasov equation via hydrodynamic limit \cite{elbar-mason-perthame-skrzeczkowski} and is also observed in the nonlocal Cahn-Hilliard equation, which can be derived from systems of interacting particles as an aggregation-diffusion equation (see~\cite{MR3913840,burger2022porous,carrillo2023nonlocal}). Furthermore, it would be of interest to investigate whether this work can be extended to the mobility case of $m(\rho)=\rho\,(1-\rho)$, as studied in the original works of Giacomin-Lebowitz \cite{MR1453735,MR1638739}. They developed a model based on a $d$-dimensional lattice gas that evolves through Kawasaki exchange dynamics, which is a Poisson process that exchanges nearest neighbors. In the hydrodynamic limit, they observed that the average occupation numbers over a small macroscopic volume element tends towards a solution of a non-local Cahn-Hilliard equation with mobility $m(\rho)=\rho(1-\rho)$.

\section{Generalised Young Measures}\label{sec:Young}

We introduce the framework of Young measures to define the solutions of the nonlocal Euler-Korteweg equation. This framework is necessary since in the usual approximation schemes we  cannot pass to the limit in the {terms} of type $f(z_{j})$ where $f$ is nonlinear and $\{z_{j}\}$ is only a {weak star} convergent sequence. The idea of Young measures is to embed the problem in a larger space and gain linearity. We write $f(z_{j}(y))=\langle f,\delta_{z_{j}(y)}\rangle$, and if $f\in C_{0}(\R^n)$, using the duality $(L^{1}(Q;C_{0}(\R^n)))^{*}=L^{\infty}_{w}(Q;\mathcal{M}(\R^n))$, Banach-Alaoglu theorem and weak-star continuity of linear operators, we can pass to the limit. The same is true if $\{f(z_j)\}$ is weakly compact in $L^1(Q)$ and $\{z_j\}$ does not grow too fast as the following theorem states:
\begin{thm}[Fundamental Theorem of Young Measures]\label{thm:fund_theorem_YM}
Let $Q\subset\R^{d}$ be a measurable set and let $z_{j}:Q\to \R^{n}$ be measurable functions such that 
\begin{equation*}
\sup_{j\in\mathbb{N}}\int_{Q}g(|z_{j}(y)|)\diff y<+\infty   
\end{equation*}
for some continuous, nondecreasing function $g:[0,+\infty)\to [0,+\infty)$ with $\lim_{t\to+\infty}g(t)=+\infty$. Then, there exists a subsequence (not relabeled) and a {weak star} measurable family of probability measures $\nu= \{\nu_{y}\}_{y\in Q}$ with the property that whenever the sequence $\{\psi(y, z_{j} (y))\}_{j\in\mathbb{N}}$ is weakly compact in $L^{1}(Q)$ for a Carathéodory function (measurable in the first and continuous in the second argument) $\psi: Q \times \R^{n}\to \R$, we have
\begin{equation*}
\psi(y,z_{j}(y))\rightharpoonup\int_{\R^{n}}\psi(y,\lambda) \diff\nu_{y}(\lambda)\quad \text{in $L^{1}(Q)$}.     
\end{equation*}
We say that the sequence $\{z_{j}\}_{j\in\mathbb{N}}$ generates the sequence of Young measures $\{\nu_{y}\}_{y\in Q}$.
\end{thm}

Note that in the above theorem, we require that the sequence $\{\psi(y,z_{j}(y))\}$ is weakly compact in $L^{1}(Q)$. This prevents the concentration effect to appear (think about the family of standard mollifiers). When we do not have weak compactness, we use the following proposition {which follows from the Banach-Alaoglu theorem and the Radon-Nikodym theorem, see~\cite{MR2805464}}. We formulate it with a distinguishment between time and space variables (that is, $Q= (0,T)\times \Omega$, $y=(t,x)$ with $t\in(0,T)$ and $x \in \Omega$) as usually in applications one has better integrability in time which results in better characterization of the resulting measure. 
\begin{proposition}\label{prop:gen_YM}
Let $f$ be a continuous function and a sequence $\{f(t,x,z_{j}(t,x))\}_{j\in\mathbb{N}}$ be bounded in $L^{p}(0,T;L^{1}(\Omega))$ with $p \geq 1$. Let $\{\nu_{t,x}\}_{t,x}$ be the Young measure generated by $\{z_{j}\}_{j}$. Then there exists a measure $m^{f}$ such that (up to a subsequence not relabelled)
\begin{equation*}
f(t,x,z_{j}(t,x))-\langle\nu_{t,x}, f\rangle \overset{\ast}{\rightharpoonup} m^{f}\quad \text{in $L^{p}(0,T;\mathcal{M}(\Omega))$ if $p>1$},   
\end{equation*}
\begin{equation*}
f(t,x,z_{j}(t,x))-\langle\nu_{t,x}, f\rangle \overset{\ast}{\rightharpoonup} m^{f}\quad \text{in $\mathcal{M}((0,T) \times \Omega)$ if $p=1$}. 
\end{equation*}
Moreover, if $p > 1$, the measure $m^f$ is absolutely continuous with respect to time: for a.e. $t \in (0,T)$, there exists measure $m^f(t,\cdot)$ such that
$$
\int_{(0,T)\times \Omega} \psi(t,x) \diff m^f(t,x) = \int_{0}^T \int_{\Omega} \psi(t,x) \, m^f(t,\diff x) \diff t. 
$$
\end{proposition}

Let us remark that by the fundamental theorem, we have $m^{f}=0$ when the sequence $\{f(z_{j})\}_{j\in\mathbb{N}}$ is weakly compact in $L^1((0,T)\times \Omega)$. We use the notation:
\begin{equation}\label{eq:notation_bar}
\overline{f}=\langle f(\lambda),\nu_{t,x}\rangle + m^{f}   
\end{equation}
to represent weak limit of $f(t,x,z_j(t,x))$. We also need the following result which allows comparing two concentration measures $m^{f_1}$ and $m^{f_2}$ for two different nonlinearities $f_1$, $f_2$. For the proof, we refer to \cite[Lemma 2.1]{MR3567640}.
\begin{proposition}\label{prop:comparison_measure}
Let $\{\nu_{t,x}\}_{(t,x)\in(0,T)\times\Omega}$ be a Young measure generated by a sequence $\{z_{j}\}_{j\in\mathbb{N}}$. If two continuous functions $f_{1}:(0,T)\times\Omega \to \R^d$ and $f_{2}:(0,T)\times\Omega \to \R^+$ satisfy $|f_{1}(z)|\le f_{2}(z)$ for every $z$, and if $\{f_{2}(z_{j})\}$ is uniformly bounded in $L^{1}((0,T)\times \Omega)$, then we have
\begin{equation*}
|m^{f_{1}}(A)|\le m^{f_{2}}(A),  
\end{equation*}
for any {B}orel set $A\subset (0,T)\times \Omega$. 
\end{proposition}
Here, $|\mu|$ is the total variation measure defined as $|\mu|(A)= \mu^+(A)-\mu^-(A)$ where $\mu^+$, $\mu^-$ are positive and negative parts of $\mu$.\\

Let us conclude with a few comments about the measure $m^{f}$ which captures concentration effects. One can describe it more precisely. The first attempts to do so by some generalizations of
the Young measures were initiated by DiPerna and Majda in the case of the incompressible Euler equations~\cite{MR877643}. Then, Alibert and Bouchitté extended the result to a more general class of nonlinearities in~\cite{MR1459885}. They proved that there exists a subsequence (not relabeled) as well as a parametrized probability measure $\nu\in L_w^\infty(Q;\mathcal{P}(\R^n))$ (which is identical with the "classical" Young measure), a non-negative measure $m\in\mathcal{M}^+(Q)$, and a parametrized probability measure $\nu^\infty\in L_w^\infty(Q,m;\mathcal{P}(\mathbb{S}^{n-1}))$ such that for any Carathéodory function $f$ such that $f(x,z)/(1+|z|)$ is bounded and uniformly continuous with respect to $z$,
\begin{equation*}
f(y,z_j(y))  \stackrel{*}{\rightharpoonup}\int_{\R^d}f(y,\lambda)d\nu_y(\lambda)  +\int_{\mathbb{S}^{n-1}}f^\infty(y,\beta) \diff \nu^\infty_y(\beta) m(y)
\end{equation*}
weakly* in the sense of measures. Here, 
\begin{equation*}
f^\infty(y,\beta):=\lim_{s\rightarrow\infty}\frac{f(y,t\beta)}{t}.
\end{equation*} 
Their result was also extended to the case when $f$ has different growth with respect to different variables, see for instance~\cite{MR3424896}.

\section{Measure-valued solutions to the nonlocal Euler-Korteweg equation}\label{sect:mv_solutions}
\subsection{Definition of the dissipative measure-valued solutions}\label{sect:def_diss_mvs}
Let us motivate the definition of a measure-valued solution by their construction. We will consider a sequence of approximating solutions $\{(\rho_{\delta}, \bu_{\delta})\}$, see Section \ref{subsect:approximation}, satisfying the estimates (uniform in $\delta$)
$$
\{\rho_{\delta}\} \mbox{ in } L^{\infty}(0,T; L^2(\Td)), \qquad \{F(\rho_{\delta})\} \mbox{ in } L^{\infty}(0,T; L^1(\Td)), \qquad \{\sqrt{\rho_{\delta}} \bu_{\delta} \} \mbox{ in } L^{\infty}(0,T; L^2(\Td)),
$$
which will be a consequence of energy inequality \eqref{est:energy_approx_after_limit_garlerkin}. As we do not have estimates on $\{\bu_{\delta}\}$ itself, we will consider in fact the sequence $\{(\rho_{\delta}, \sqrt{\rho_{\delta}}\,\bu_{\delta})\}$. Up to a subsequence, we have as $\delta \to 0$ 
\begin{equation}\label{eq:weak_comp_rhodelta_rhou}
\rho_{\delta} \weaks \rho \mbox{ in } L^{\infty}(0,T; L^2(\Td)) \qquad \qquad 
\sqrt{\rho_{\delta}} \bu_{\delta} \weaks \overline{\sqrt{\rho} \bu} \mbox{ in } L^{\infty}(0,T; L^2(\Td)),
\end{equation}
where $\overline{\sqrt{\rho} \bu}$ is a {\it definition} of a weak limit of $\sqrt{\rho_{\delta}} \bu_{\delta}$. Let $\{\nu_{t,x}\}$ be the Young measure generated by this sequence as in Theorem \ref{thm:fund_theorem_YM}. We will use dummy variables $(\lambda_1,\lambda')\in\R^+\times\R^d$ when integrating with respect to $\nu_{t,x}$:
\begin{equation}\label{eq:YM_how_to_integrate}
\langle F(\lambda_1,\lambda'),\nu_{t,x}\rangle:=\int_{\R^+\times\R^d}F(\lambda_1,\lambda') \diff \nu_{x,t}(\lambda_1,\lambda'),
\end{equation}
with $\lambda_1$ representing $\rho$ variable and $\lambda'$ as representing $\overline{\sqrt{\rho}\bu}$ variable. In terms of Young measures we write weak convergence \eqref{eq:weak_comp_rhodelta_rhou} as 
\begin{equation}\label{eq:id_wl_1}
{\rho}= \langle\lambda_1,\nu\rangle, \qquad \overline{\sqrt{\rho} \bu} = \langle\lambda', \nu \rangle,
\end{equation} 
as there is no concentration measure because of integrability in $L^2((0,T)\times \Td)$. Using notation~\eqref{eq:notation_bar} we represent weak limits (as $\delta \to 0$) of all the terms that should appear in the weak formulation and the energy
\begin{equation}\label{eq:id_wl_2}
\overline{\rho^2}
=\langle\lambda_1^2,\nu\rangle+m^{\rho^{2}},
\end{equation}
\begin{equation}\label{eq:id_wl_3}
\overline{\rho \bu}=\langle\sqrt{\lambda_1}\lambda',\nu\rangle,
\end{equation}
\begin{equation}\label{eq:id_wl_4}
\overline{\rho\bu\otimes \bu}=\langle\lambda'\otimes\lambda',\nu\rangle+m^{\rho \bu\otimes\bu},
\end{equation}
\begin{equation}\label{eq:id_wl_5}
\overline{\rho|\bu|^2}=\langle|\lambda'|^2,\nu\rangle+m^{\rho|\bu|^{2}},
\end{equation}
\begin{equation}\label{eq:id_wl_8}
\overline{F(\rho)} = \langle F(\lambda_1),\nu\rangle + m^{F(\rho)}
\end{equation}
\begin{equation}\label{eq:id_wl_9}
\overline{\rho\, F'(\rho)}  = \langle \lambda_1 F'(\lambda_1),\nu\rangle + m^{\rho F'(\rho)},
\end{equation}
\begin{equation}\label{eq:id_wl_10}
\overline{p(\rho)} = \overline{\rho\, F'(\rho)} - \overline{F(\rho)} + \frac{1}{2\eta^2} \, \overline{\rho^2},
\end{equation}
where $p(\rho):= \rho F'(\rho)-F(\rho)+\f{\rho^{2}}{2\eta^{2}}$.\\

Moreover, we will identify weak limits of several nonlinearities which will be used in this work. By linearity of weak limits, we have the following identities:
\begin{equation}\label{eq:def_convolution_rho_squared}
\overline{\int_{\Td}\omega_{\eta}(y)|\rho(x)-\rho(x-y)|^{2}\diff y} = \overline{\rho^{2}} + \overline{\rho^{2}} \ast \omega_{\eta} -2\,\rho \, \omega_{\eta}\ast\rho
\end{equation}
Similarly, for all bounded $P: (0,T)\times [0,+\infty) \to \R^+$ and $\Bu: (0,T)\times \Td \to \R^d$ we have
\begin{equation}\label{eq:def_square_diff_densities}
\overline{|\rho - \Rho|^2} = \overline{\rho^2} + \Rho^2 - 2\rho\, \Rho
\end{equation}
\begin{equation}\label{eq:def_rel_kin_energy}
\overline{\rho|\bu-\Bu|^{2}} = \overline{\rho\,|\bu|^2} + \rho\, |\Bu|^2 - 2\,\overline{\rho \bu} \cdot U,
\end{equation}
\begin{equation}\label{eq:def_rel_kin_energy_tensor}
\overline{\rho(\bu-\Bu)\otimes(\bu-\Bu)} =  \langle (\lambda' - \sqrt{\lambda_1}\Bu) \otimes (\lambda' - \sqrt{\lambda_1}\Bu), \nu_{t,x}\rangle + m^{\rho \bu\otimes \bu},
\end{equation}
\begin{equation}\label{eq:nonlocal_term_with_P}
\begin{split}
&\overline{\int_{\Td}\omega_{\eta}(y)|(\rho-\Rho)(x)-(\rho-\Rho)(x-y)|^{2}\diff y}= \overline{\int_{\Td} \omega_{\eta}(y)|\rho(x) - \rho(x-y)|^2 \diff y} \, + \\ & \qquad + 
\int_{\Td} \omega_{\eta}(y)|P(x) - P(x-y)|^2 \diff y
- 2 \int_{\Td} \omega_{\eta}(y) (P(x) - P(x-y)) (\rho(x) - \rho(x-y) \diff y,
\end{split}
\end{equation}
\begin{equation}\label{eq:def_rel_potenial_rho_Rho}
\overline{F(\rho|\Rho)} := \overline{F(\rho)}-F(\Rho)-F'(\Rho)(\rho-\Rho), \qquad \overline{p(\rho|\Rho)} : =\overline{p(\rho)}-p(\Rho)-p'(\Rho)(\rho-\Rho)
\end{equation}
where nonlinearities are defined as
\begin{equation}\label{eq:F(rho-Rho)}
F(\rho|\Rho)=F(\rho)-F(\Rho)-F'(\Rho)(\rho-\Rho), \qquad  p(\rho|\Rho) = p(\rho) - p(\Rho) - p'(\Rho)(\rho-\Rho).     
\end{equation}

Now, we define measure-valued solutions by {\it inverting} this discussion. \\

\begin{Def}[Measure-valued solution]\label{def:dissipative_measure} We say that $(\rho, \overline{\sqrt{\rho}\bu}, \nu,m)$ where
$$
\nu = \{\nu_{t,x}\} \in L^{\infty}_{\text{weak}}((0,T)\times\Td;\mathcal{P}([0,+\infty)\times\mathbb{R}^{d}))
$$
$$
\rho =  \langle\lambda_1,\nu\rangle = \int_{\R^+\times\R^d} \lambda_1 \diff \nu_{x,t}(\lambda_1,\lambda')  \in L^{\infty}(0,T; L^2(\Td)), 
$$
$$
\overline{\sqrt{\rho}  \bu}  =  \langle\lambda' ,\nu\rangle = \int_{\R^+\times\R^d} \lambda' \diff \nu_{x,t}(\lambda_1,\lambda')  \in L^{\infty}(0,T; L^2(\Td)),
$$
$$
m = \left(m^{\rho^{2}}, m^{\rho \bu\otimes\bu}, m^{\rho|\bu|^{2}}, m^{F(\rho)}, m^{\rho F'(\rho)}\right)
$$
with 
\begin{align*}
&m^{\rho^2}, m^{\rho|\bu|^2}, m^{F(\rho)} \in L^{\infty}((0,T);\mathcal{M}^{+}(\Td)),  \qquad 
&&m^{\rho F'(\rho)} \in L^{\infty}((0,T);\mathcal{M}(\Td)), \\
&m^{\rho \bu\otimes\bu} \in L^{\infty}((0,T); \mathcal{M}(\Td)^{d\times d})
\end{align*}
and
\begin{equation}\label{eq:inequality_2_conc_measures}
|m^{\varrho \bu\otimes \bu}| \leq m^{\varrho |\bu|^2}
\end{equation}
\begin{equation}\label{eq:inequality_2_conc_measures_func_p_F}
|m^{\rho\,F'(\rho)}| \leq C_F\, m^{F(\rho)} + C_F \,m^{\rho^2}, \mbox{$C_F$ defined in \eqref{eq:bound_notrel_p_F}} 
\end{equation}

is a measure-valued solution of~\eqref{eq:EK1}-\eqref{eq:EK2} with initial data $({\rho^0},{\bu^0})$ if for every $\psi\in C^1_c([0,T)\times\Td;\R)$, $\phi\in C^1_c([0,T)\times\Td;\R^d)$ it holds that
\begin{equation}\label{eq:weak_sol_EK_approx_eq1}
\int_0^T \int_{\Td} \partial_t\psi \, {\rho}+\f{1}{\eps}\nabla\psi\cdot \overline{\rho\bu} \diff x \diff t 
+\int_{\Td}\psi(x,0){\rho^0} \diff x =0,
\end{equation}
\begin{equation}\label{eq:weak_sol_EK_approx_eq2}
\begin{split}
\int_0^T \int_{\Td}\partial_t\phi\cdot\overline{\rho \bu}+\f{1}{\eps}\nabla\phi : \overline{\rho \bu\otimes \bu}-\f{1}{\eps^{2}}\phi\cdot\overline{\rho\bu}+\f{1}{\eps}\DIV\phi\, \overline{p(\rho)}&+\f{1}{\eps\eta^{2}}\phi\cdot {\rho \nabla\omega_{\eta}\ast\rho} \diff x \diff t
\\
&   +\int_{\Td}\phi(x,0)\cdot {\rho^0}\, {\bu^0} \diff x=0,
\end{split}
\end{equation}
where $p(\rho)=\rho F'(\rho)-F(\rho)+\f{\rho^{2}}{2\eta^{2}}$ and all the terms are defined in \eqref{eq:id_wl_1}--\eqref{eq:id_wl_10}.  
\end{Def}
\begin{Def}[nonlinear functions]\label{def:nonlin_fun_mvs}
Given a measure-valued solution $(\rho, \overline{\sqrt{\rho}\bu}, \nu,m)$ and bounded $P: (0,T)\times [0,+\infty) \to \R^+$, $\Bu: (0,T)\times \Td \to \R^d$, we define nonlinear quantities
$$
\overline{\int_{\Td}\omega_{\eta}(y)|\rho(x)-\rho(x-y)|^{2}\diff y}, \qquad \overline{|\rho - \Rho|^2}, \qquad \overline{\rho|\bu-\Bu|^{2}}, \qquad \overline{F(\rho|\Rho)}, \qquad \overline{p(\rho|\Rho)},
$$
$$
\overline{\rho(\bu-\Bu)\otimes(\bu-\Bu)}, \qquad \overline{\int_{\Td}\omega_{\eta}(y)|(\rho-\Rho)(x)-(\rho-\Rho)(x-y)|^{2}\diff y}
$$
by formulas \eqref{eq:def_convolution_rho_squared}--\eqref{eq:def_rel_potenial_rho_Rho}.
\end{Def}
\begin{Def}[energy]
Given a measure-valued solution $(\rho, \overline{\sqrt{\rho}\bu}, \nu,m)$ for a.e. $t \in (0,T)$ we define the energy as
\begin{equation*}
E_{mvs}(t):=\int_{\Td}\f{1}{2}\overline{\rho|\bu|^{2}}+\overline{F(\rho)}\diff x+\f{1}{4\eta^{2}}\int_{\Td} \overline{\int_{\Td}\omega_{\eta}(y)|\rho(x)-\rho(x-y)|^{2}\diff y} \diff x,   
\end{equation*}
where the nonlocal term is defined by \eqref{eq:def_convolution_rho_squared}. We also define
\begin{equation*}
E_0:=\int_{\Td}\frac{1}{2}{\rho^0}|u_0|^2(x)+F({\rho^0})\diff x+\f{1}{4\eta^{2}}\int_{\Td}\int_{\Td}\omega_{\eta}(y)|{\rho^0}(x)-{\rho^0}(x-y)|^{2}\diff x\diff y.   
\end{equation*}
\end{Def}
This energy is well-defined because, by Proposition \ref{prop:gen_YM}, a concentration measure $m\in L^{\infty}(0,T; \mathcal{M}(\Td))$ admits disintegration $\diff m(t,x) = m(t,\diff x) \diff t$ where $m(t,\cdot)$ is a well-defined measure on $\Td$ for a.e. $t \in (0,T)$.\\ 

We now introduce two properties which allows to select the \textit{right} measure-valued solutions.
\begin{Def}[Dissipativite measure-valued solution]\label{def:diss_mvs}
We say that a measure-valued solution $(\rho, \sqrt{\rho}\bu, \nu,m)$ is dissipative if
\begin{equation}\label{EEmvsenergy}
\begin{aligned}
E_{mvs}(t)+\f{1}{\eps^{2}}\int_{0}^{t}\int_{\Td}\overline{\rho|\bu|^{2}}\diff x\diff t\leq E_0
\end{aligned}
\end{equation}
for almost every $t\in(0,T)$. 
\end{Def}

\begin{Def}[Poincaré inequality]\label{def:admi_mvs}
A measure-valued solution $(\rho, \overline{\sqrt{\rho}\bu}, \nu,m)$ with initial condition ${\rho^0}$ satisfies the nonlocal Poincaré inequality if for a.e. $t \in (0,T)$ and all bounded $\Rho:\Omega_T \to [0,+\infty)$ such that $(\Rho)_{\Td} = ({\rho^0})_{\Td}$ we have
\begin{equation}\label{eq:nonloc_Poinc_stronger_bar}
\int_{\Td} \overline{|\rho - P|^2} \diff x \leq  \frac{C_P}{4\eta^2}\, \int_{\Td} \overline{\int_{\Td}\omega_{\eta}(y)|(\rho-\Rho)(x)-(\rho-\Rho)(x-y)|^{2}\diff y} \diff x .
\end{equation}
where the constant $C_P$ is given by Lemma \ref{lem:Poincaré_nonlocal_H1_L2}.
\end{Def}
Let us remark that in Lemma \ref{lem:rel_entr_nonneg}, we will prove that any measure-valued solution satisfies
$$
\int_{\Td} {|\rho - P|^2} \diff x \leq  \frac{C_P}{4\eta^2}\, \int_{\Td} \overline{\int_{\Td}\omega_{\eta}(y)|(\rho-\Rho)(x)-(\rho-\Rho)(x-y)|^{2}\diff y} \diff x .
$$
which is a weaker version of \eqref{eq:nonloc_Poinc_stronger_bar}. Nevertheless, \eqref{eq:nonloc_Poinc_stronger_bar} will be necessary to estimate several terms appearing in the application of the relative entropy method in Section \ref{sect:conv_rel_entr}. Let us also point out that similar Poincaré-type inequalities are usually assumed for measure-valued solutions to several different PDEs, see for instance \cite[eq. (2.23)]{MR3567640}.\\

We conclude with a simple observation concerning the energy.
\begin{lem}\label{lem:nonneg_the_energy_one}
The energy $E_{mvs}$ defined by \eqref{EEmvsenergy} is nonnegative.
\end{lem}
\begin{proof}
The lemma seems to be trivial from the point of view of our discussion about weak limits at the beginning of this section. However, the measure-valued solution is defined by Definition \ref{def:dissipative_measure} so that we can argue only using Definitions \ref{def:dissipative_measure} and \ref{def:nonlin_fun_mvs}. Clearly, $\f{1}{2}\overline{\rho|\bu|^{2}}$ and $\overline{F(\rho)}$ are nonnegative so that we only have to study the nonlocal term. By \eqref{eq:def_convolution_rho_squared}, 
$$
\int_{\Td} \overline{\int_{\Td}\omega_{\eta}(y)|\rho(x)-\rho(x-y)|^{2}\diff y} {\,\diff x} = 2 \int_{\Td} \overline{\rho^{2}} {\,\diff x} -2 \int_{\Td} \rho \, \omega_{\eta}\ast\rho {\,\diff x}.
$$
By Cauchy-Schwarz and Young convolution inequalities:
$$
2 \int_{\Td} \rho \, \omega_{\eta}\ast\rho \diff x \leq 2 \int_{\Td} \rho^2 \diff x.
$$
Using Jensen's inequality (measure $\nu_{t,x}$ is the probability measure with respect to both coordinates)
\begin{equation}\label{eq:estimate_rho_sq_sq_rho}
\int_{\Td} \rho^2 \diff x = \int_{\Td} \langle \lambda_1, \nu_{t,x} \rangle^2 \diff x \leq \int_{\Td} \langle \lambda_1^2, \nu_{t,x} \rangle \diff x \leq \int_{\Td} {\overline{\rho^2}} \diff x
\end{equation}
so that the nonlocal term is nonnegative.
\end{proof}

\subsection{The approximating system}\label{subsect:approximation}

To construct a measure-valued solution we use a method as outlined in~\cite[Section 5.5]{MR1409366}, see also~\cite{Debiec,MR2182484}. This is a fairly standard procedure based on regularizing density by a positive parameter
\begin{equation}\label{eq:init_cond_approx_delta}
{\rho^0_{\delta}}:={\rho^0}+\delta,\quad {\rho^0}\in C^{1}(\Td),\, {\rho^0}>0, \quad { \bu^{0}_{\delta}}(x):={ \bu^{0}}(x)\in W^{3,2}(\Td)^{d},     
\end{equation}
which makes the density $\rho_{\delta}$ globally bounded from below. We will only discuss the main steps and for the full presentation, we refer to \cite[Section 5.5]{MR1409366}.\\

We work in $W^{3,2}(\Td)^d$ (but for dimensions $d$ higher than 3, we need to work even in $W^{1+d,2}(\Td)$) because of the embedding $W^{3,2}(\Td) \subset C^1(\Td)$ which will be important for certain estimates. We use notation $(\!(\cdot,\cdot)\!)$ for the standard scalar product in $W^{3,2}(\Td)^d$. By \cite[Appendix, Theorem 4.11]{MR1409366}, we take $\{\bwi\}$ to be an orthonormal basis of $W^{3,2}(\Td)^{d}$ which are $C^{\infty}(\Td)^{d}$ functions. Finally, we define $\Pi^{N}$ to be the projection operator into $\Span\{\bm{\omega_{1}},..., \bm{\omega_{N}}\}$ which satisfies $\norm{\Pi^{N}\bu}_{W^{3,2}}\le \norm{\bu}_{W^{3,2}}$ and $\norm{\Pi^{N}\bu}_{L^2}\le \norm{\bu}_{L^2}$. \\

We will find solution $(\rho_\delta,\bu_\delta)$ such that 
\begin{equation}\label{est:regularity_approx_delta}
\begin{split}
&\rho_\delta\in L^{\infty}((0,T)\times\Td)\cap L^{2}(0,T;W^{1,2}(\Td)),\quad \f{\p\rho_{\delta}}{\p t}\in L^{2}((0,T)\times\Td)\\
&\bu_\delta\in L^{\infty}(0,T;W^{3,2}(\Td)),\quad \f{\p\bu_{\delta}}{\p t}\in L^{2}((0,T)\times\Td), 
\end{split}
\end{equation}
to the following problem: for all $\psi\in C^1_c([0,T)\times\Td;\R)$, $\phi\in C^1_c([0,T)\times\Td;\R^d)$ it holds that
\begin{equation}\label{eq:approx_delta_rho}
\int_0^T\int_{\Td}\partial_t\psi \rho_{\delta}+\f{1}{\eps}\nabla\psi\cdot\rho_{\delta}\bu_{\delta} \diff x \diff t +\int_{\Td}\psi(x,0){\rho^0_{\delta}} \diff x=0,
\end{equation}
\begin{equation}\label{eq:approx_delta_u}
\begin{split}
&\int_0^T\int_{\Td}\partial_t\phi\cdot\rho_{\delta}\bu_{\delta}+\f{1}{\eps}\nabla\phi : \rho_{\delta} \bu_{\delta}\otimes \bu_{\delta}-\f{1}{\eps^{2}}\phi\cdot\rho_{\delta}\bu_{\delta}+\f{1}{\eps}\DIV\phi \, p(\rho_{\delta})
\diff x \diff t\\
& \qquad +\int_0^T\int_{\Td} \f{1}{\eps\eta^{2}}\phi\cdot\rho_{\delta} \nabla\omega_{\eta}\ast\rho_{\delta} \diff x \diff t+\int_{\Td}\phi(x,0)\cdot {\rho^0_{\delta}}{ \bu^{0}_{\delta}} \diff x =\delta\int_{0}^{T}(\!(\bu_{\delta},\phi)\!)\diff t.
\end{split}
\end{equation}
To find the solution to \eqref{eq:approx_delta_rho}--\eqref{eq:approx_delta_u}, we use the method of Galerkin approximations. We look for $\bu^N$ of the form
$$
\bu^{N}=\sum_{j=1}^{N}c_{j}^{N}(t)\,\bwj
$$
solving
\begin{equation}\label{eq:approx_delta_galerkin_1}
\f{\p \rho^{N}}{\p t}+\f{1}{\eps}\DIV(\rho^{N}\bu^{N}) = 0,
\end{equation}
\begin{equation}\label{eq:approx_delta_galerkin_2}
\begin{split}
\int_{\Td}\left(\rho^{N}\p_{t}\bu^{N}+\f{1}{\eps}\rho^{N} \bu^{N} \nabla\bu^{N}+\f{1}{\eps^{2}}\rho^{N}\bu^{N}+\f{1}{\eps}\nabla p(\rho^{N})-\f{1}{\eps\eta^{2}}\rho^{N} \nabla\omega_{\eta}\ast\rho^{N}\right)\cdot \bm{\omega_{i}} \diff x\,+\\+\, \delta (\!(\bu^{N},\bm{\omega_{i}})\!)=0,
\end{split} 
\end{equation}
for $i=1,...,N$ with initial conditions $\rho^{N}(0)={\rho^0_{\delta}}$, $\bu^{N}(0)=\Pi^{N}{ \bu^{0}_{\delta}}$.\\

The proof of existence to \eqref{eq:approx_delta_galerkin_1}--\eqref{eq:approx_delta_galerkin_2} follows 3 steps: using a fixed point argument to prove the existence on a small interval, deriving a priori estimates on this interval, extending the procedure on the whole interval. The crucial point is the lower bound on $\rho^N$ in terms of $\delta$. This is obtained by the method of characteristics. Indeed,
\begin{equation}\label{eq:rho_estimate_from_below}
\rho^N(t,x) \geq \essinf_{x \in \Td} {\rho^0_{\delta}} \,  \exp\left(-\f{1}{\eps}\int_0^T \| \text{div} \, \bu^N\|_{\infty} \diff t \right) \geq \delta \, \exp\left(-\f{1}{\eps}\int_0^T \|\bu^N\|_{W^{3,2}} \diff t \right)
\end{equation}
by the well-known formula for the continuity equation. On the other hand, thanks to the regularizing term, $\|\bu^N\|_{L^2(0,T; W^{3,2}(\Td))} \leq \frac{C}{\delta}$. This gives uniform lower (and also upper) bound on $\rho^N$ and allows us to look at \eqref{eq:approx_delta_galerkin_2} as a system of ODEs. We refer to~\cite[Section 5.5]{MR1409366} and omit the details. We obtain the following lemma:
\begin{lem}\label{lem:estimates_Galerkin}
For fixed $N$, there exists a solution to~\eqref{eq:approx_delta_galerkin_1}--\eqref{eq:approx_delta_galerkin_2} such that $\rho^{N}\in C^{1}([0,T]\times\Td)$, $\bu^{N}\in C^{1}([0,T];W^{3,2}(\Td)^{d})$. Moreover, we have the energy estimate: for all times $\tau \in [0,T]$
\begin{equation}\label{est:energy_approx_garlerkin}
\begin{split}
&\int_{\Td}\f{1}{2}\rho^{N}|\bu^{N}|^{2}+F(\rho^{N})\diff x+\f{1}{4\eta^{2}}\int_{\Td}\int_{\Td}\omega_{\eta}(y)|\rho^{N}(x)-\rho^{N}(x-y)|^{2}\diff x\diff y\\ 
&+ \delta\int_{0}^{\tau} \| u^N \|_{W^{3,2}}^2 \diff t +\f{1}{\eps^{2}}\int_{0}^{\tau}\int_{\Td}\rho^{N}|\bu^{N}|^{2}\diff x { \, \diff t} \le \\
&\le \int_{\Td}\f{1}{2}{\rho^0_{\delta}}|{ \bu^{0}}|^{2}+F({\rho^0_{\delta}})\diff x+\f{1}{4\eta^{2}}\int_{\Td}\int_{\Td}\omega_{\eta}(y)|{\rho^0_{\delta}}(x)-{\rho^0_{\delta}}(x-y)|^{2}\diff x\diff y,
\end{split}
\end{equation}
as well as the following estimates
\begin{align}
&\rho^{N}(t,x)\ge C\left(\f{1}{\delta}\right) \label{eq:bounds_delta_1}\\
&\norm{\rho^{N}}_{L^{\infty}((0,T)\times\Omega)}+\int_{0}^{\tau}\norm{\p_{t}\rho^{N}}_{L^{2}(\Td)}^{2}{\diff t}+\int_{0}^{\tau}\norm{\nabla\rho^{N}}_{L^{2}(\Td)}^{2}{\diff t}\le C\left(\f{1}{\delta}\right) \label{eq:bounds_delta_2},\\
&\int_{0}^{T}\norm{\p_{t}\bu^{N}}_{L^{2}(\Td)}^{2}{\diff t}+\delta\norm{\bu^{N}}_{L^{2}((0,T);W^{3,2}(\Td)^{d})}\le C\left(\f{1}{\delta}\right) \label{eq:bounds_delta_3},
\end{align}
where $C\left(\f{1}{\delta}\right)$ is a constant depending on $\f{1}{\delta}$ and other fixed parameters (like $\varepsilon$).
\end{lem}
\begin{proof}
The energy estimate follows by testing \eqref{eq:approx_delta_galerkin_2} by $\bu^N$ (in the Galerkin sense: we multiply \eqref{eq:approx_delta_galerkin_2} by $c_i^N$ and sum for $i = 1, ..., N$) {and sum it with \eqref{eq:approx_delta_galerkin_1} multiplied by $\frac{1}{2}\,|\bu^{N}|^{2}$}. Estimate \eqref{eq:bounds_delta_1} follows from the characteristics as explained in \eqref{eq:rho_estimate_from_below}. Similarly, we obtain the upper bound. Concerning the estimates on derivatives of $\rho^N$, they follow by differentiating the formula from the method of characteristics and using the bound $\|\bu^N\|_{L^2(0,T; W^{3,2}(\Td))} \leq \frac{C}{\delta}$. Finally, \eqref{eq:bounds_delta_3} is a consequence of testing \eqref{eq:approx_delta_galerkin_2} by $\partial_t\bu^N$.
\end{proof}
Using the estimates in Lemma \ref{lem:estimates_Galerkin}, up to a subsequence, we can pass to the limit $N \to \infty$
\begin{align*}
    &\rho^{N}\to \rho_{\delta} \quad \text{strongly in $L^{2}((0,T)\times\Td)$},\\
    &\bu^{N}\to \bu_{\delta} \quad \text{strongly in $L^{2}((0,T) \times \Td)^{d}$}
\end{align*}
(the convergence holds even in better spaces). We also have an energy inequality:
\begin{equation}\label{est:energy_approx_after_limit_garlerkin}
\begin{split}
&\int_{\Td}\f{1}{2}\rho_{\delta}|\bu_{\delta}|^{2}+F(\rho_{\delta})\diff x+\f{1}{4\eta^{2}}\int_{\Td}\int_{\Td}\omega_{\eta}(y)|\rho_{\delta}(x)-\rho_{\delta}(x-y)|^{2}\diff x\diff y\\ 
&+ \delta\int_{0}^{\tau} \| u_{\delta} \|_{W^{3,2}}^2 \diff t +\f{1}{\eps^{2}}\int_{0}^{\tau}\int_{\Td}\rho_{\delta}|\bu_{\delta}|^{2}\diff x\le \\
&\le \int_{\Td}\f{1}{2}{\rho^0_{\delta}}|{ \bu^{0}}|^{2}+F({\rho^0_{\delta}})\diff x+\f{1}{4\eta^{2}}\int_{\Td}\int_{\Td}\omega_{\eta}(y)|{\rho^0_{\delta}}(x)-{\rho^0_{\delta}}(x-y)|^{2}\diff x\diff y,
\end{split}
\end{equation}
 This concludes the proof of existence of $(\rho_{\delta},\bu_{\delta})$ satisfying~\eqref{eq:approx_delta_rho}--\eqref{eq:approx_delta_u}. 

\subsection{Existence of dissipative measure-valued solutions}
It remains to pass to the limit $\delta\to 0$ in \eqref{eq:approx_delta_rho}--\eqref{eq:approx_delta_u}. First we gather some uniform bounds in $\delta$, being a simple consequence of \eqref{eq:bounds_delta_1} and \eqref{est:energy_approx_after_limit_garlerkin}, in the following lemma: 
\begin{lem}\label{lem:estimates_just_delta}
Let $(\rho^{\delta},\bu^{\delta})$ be weak solutions of~\eqref{eq:approx_delta_rho}--\eqref{eq:approx_delta_u} as constructed above. Then, there exists a constant $C>0$ independent of $\delta$ such that
\begin{align*}
&\rho_{\delta}\ge 0\quad \text{a.e. in $(0,T)\times\Td$},\\
&\norm{\sqrt{\rho_{\delta}}\bu_{\delta}}_{L^{\infty}(0,T;L^{2}(\Td))}\le C,\quad \norm{F(\rho_{\delta})}_{L^{\infty}(0,T;L^{1}(\Td))}\le C,\quad \norm{\rho_{\delta}}_{L^{\infty}(0,T;L^{2}(\Td))}\le C,\\
&\delta\norm{\bu_{\delta}}_{L^{2}(0,T;W^{3,2}(\Td))}^{2}\le C,\\
&\norm{\p_{t}\rho_{\delta}}_{L^{2}(0,T;(W^{1,4}(\Td))')}\le C. 
\end{align*}
\end{lem}
{The last estimate is a consequence of the splitting
$$
\p_{t}\rho_{\delta}=-\f{1}{\eps}\DIV(\sqrt{\rho_{\delta}}\sqrt{\rho_{\delta}}\bu_{\delta})
$$
so that thanks to the uniform estimates on $\sqrt{\rho_{\delta}}$ in $L^{\infty}(0,T; L^4(\Td))$ and $\sqrt{\rho_{\delta}}\bu_{\delta}$ in $L^{\infty}(0,T; L^2((\Td))$ we have $\sqrt{\rho_{\delta}}\sqrt{\rho_{\delta}}\bu_{\delta} \in L^{\infty}(0,T; L^\frac{4}{3}(\Td))$}.\\

Now, the proof of the existence of dissipative measure-valued solution follows the method described at the beginning of Section \ref{sect:def_diss_mvs}. By Lemma \ref{lem:estimates_just_delta}, we have sufficient estimates to have convergence \eqref{eq:weak_comp_rhodelta_rhou} which allows us to define the Young measure $\{\nu_{t,x}\}$ as in \eqref{eq:YM_how_to_integrate}--\eqref{eq:id_wl_1}. Then, the representations formulas for weak limits of nonlinearities \eqref{eq:id_wl_2}--\eqref{eq:id_wl_5} are a consequence of Lemma \ref{prop:gen_YM} and the estimate on $\norm{\sqrt{\rho_{\delta}}\bu_{\delta}}_{L^{\infty}(0,T;L^{2}(\Td))}$ which guarantees that all of the considered quantities are at least in $L^{\infty}(0,T; L^1(\Td))$. Note that $m^{\rho \bu} = 0$ because we have a uniform bound $\| \rho_{\delta} \bu_{\delta}\|_{L^{\infty}(0,T; L^{\frac{4}{3}}(\Td))} \leq C$. Next, \eqref{eq:id_wl_8} follows from the estimate on $\norm{F(\rho_{\delta})}_{L^{\infty}(0,T;L^{1}(\Td))}$. Here, the measure $m^{F(\rho)}$ is nonnegative because $F=F_1 + F_2$ where $F_1 \geq 0$ while $F_2$ is bounded so that the only concentration effect can arise from $F_1$. Similarly, by Assumption \ref{ass:potentialF}, $\norm{\rho_{\delta}\,F'(\rho_{\delta})}_{L^{\infty}(0,T;L^{1}(\Td))} \leq C$ so that \eqref{eq:id_wl_9} follows. Finally, \eqref{eq:id_wl_10} is a consequence of the linearity and uniqueness of weak limits. This allows to pass to the limit $\delta\to0$ in almost all of the terms in formulation \eqref{eq:approx_delta_rho}--\eqref{eq:approx_delta_u}.\\

Concerning the regularizing term on the (RHS) of \eqref{eq:approx_delta_u}, we observe that
$$
\left|\delta\int_{0}^{T}(\!(\bu_{\delta},\phi)\!)\diff t\right| \leq
\delta\, \| \bu_{\delta}\|_{L^2((0,T);W^{3,2}(\Td))} \,  \|  \phi \|_{L^2((0,T);W^{3,2}(\Td))} \leq C \sqrt{\delta} \, \|  \phi \|_{L^2((0,T);W^{3,2}(\Td))} \to 0. 
$$
When it comes to the nonlocal terms, we observe that we can identify their weak limits because the convolution upgrades a weak convergence to the strong one. More precisely, if $\rho_{\delta} \weaks \rho$ in $L^{\infty}(0,T); L^2(\Td))$, then $\rho_{\delta} \ast \omega_{\eta} \to \rho \ast \omega_{\eta}$ in $L^p(0,T; L^p(\Td))$ strongly, for all $1 \leq p < \infty$. This follows by the Lions-Aubin lemma and a standard subsequence argument as the sequence $\{\rho_{\delta} \ast \omega_{\eta}\}_{\delta}$ has uniformly bounded derivatives in the spatial derivatives while its time derivative is bounded in some negative Sobolev space by Lemma \ref{lem:estimates_just_delta}. \\ 

Concerning \eqref{eq:inequality_2_conc_measures}, we notice that it is a consequence of the inequality
$$
\left| \lambda' \otimes \lambda' \right| = \left( \sum_{i,j=1}^d \left( \lambda_i' \lambda_j' \right)^2 \right)^{1/2} = \sum_{i=1}^d |\lambda_i'|^2 = |\lambda'|^2
$$
and Lemma \ref{prop:comparison_measure}. Similarly, \eqref{eq:inequality_2_conc_measures_func_p_F} follows from Proposition \ref{prop:comparison_measure} and inequality \eqref{eq:bound_notrel_p_F}. \\

Next, the constructed measure-valued solution is dissipative in the sense of Definition \ref{def:diss_mvs} because we can pass to the limit in \eqref{est:energy_approx_after_limit_garlerkin} using identified weak limits (rigorously, one multiplies \eqref{est:energy_approx_after_limit_garlerkin} with a nonnegative test function of time, passes to the limit and then performs a standard localization argument).\\

Finally, the constructed solution satisfies Poincaré inequality as in Definition \ref{def:admi_mvs}. Indeed, by Lemma \ref{lem:Poincaré_nonlocal_H1_L2} we have for all bounded and nonnegative $\varphi:[0,T] \to [0,\infty)$
$$
\int_0^T  \int_{\Td}\varphi(t)  |(\rho_{\delta} - \Rho) - \delta |^2 { \, \diff x\diff t}\leq  \frac{C_P}{4\eta^2}\, \int_0^T  \int_{\Td} \int_{\Td}  \varphi(t) {| (\rho_{\delta} - \Rho)(x) - (\rho_{\delta} - \Rho)(x-y)|^2} \omega_\eta(y) \diff x \diff y \diff t 
$$
because $(\rho_{\delta} - \Rho)_{\Td} = \delta$. The (LHS) can be written as
$$
\int_0^T  \int_{\Td} \varphi(t)  |(\rho_{\delta} - \Rho) - \delta |^2
 {\, \diff x\diff t}=
\int_0^T \int_{\Td} \varphi(t) \left((\rho_{\delta}- \Rho)^2 + \delta^2 - 2\, \delta \, (\rho_{\delta}- \Rho)\right) {\diff x\diff t}.
$$
As $\rho_{\delta} - \Rho$ is bounded in $L^{\infty}(0,T;L^2(\Td))$, the last two terms vanish in the limit $\delta \to 0$. Finally, the term $(\rho_{\delta}- \Rho)^2$ has weak limit $\overline{\rho^2} + \Rho^2 - 2\,\Rho \, \rho$ which is exactly $\overline{(\rho - \Rho)^2}$, cf. \eqref{eq:def_square_diff_densities}. Similarly, we consider the term on the (RHS) so that we obtain
$$
\int_0^T  \int_{\Td}\varphi(t)  \overline{|\rho - \Rho |^2}{ \, \diff x\diff t} \leq  \frac{C_P}{4\eta^2}\, \int_0^T \varphi(t) \int_{\Td} \overline{ \int_{\Td} \omega_\eta(y) | (\rho - \Rho)(x) - (\rho- \Rho)(x-y)|^2 \diff y} \diff x \diff t. 
$$
As this inequality holds for all $\varphi$, we conclude the proof.

\section{Classical solutions to the nonlocal Cahn-Hilliard equation}\label{sect:class_sol}

To prove the convergence of the measure-valued solution of the nonlocal Euler-Korteweg to a solution of the Cahn-Hilliard equation, we use arguments similar to weak-strong uniqueness. Therefore, we study below the classical solutions of the nonlocal Cahn-Hilliard equation.
More precisely, we consider the equation \eqref{eq:CHNL1}--\eqref{eq:CHNL2}. The initial condition is a smooth positive function, more precisely we consider for some $\alpha,\sigma>0$
\begin{equation}\label{eq:init_cond}
\rho(0,x)={\rho^0}(x), \quad {\rho^0}\in C^{2+\alpha}(\Td), \quad {\rho^0}(x)\ge\sigma\quad \forall x\in\Td.     
\end{equation}
We also suppose that $F \in C^4$ which is required by the parabolic regularity theory exploited in Lemma \ref{lem:class_sol_approx}. Equations~\eqref{eq:CHNL1}-\eqref{eq:CHNL2} can be rewritten as
\begin{equation}\label{eq:CHNL3}
\p_{t}\rho-\Delta(\phi(\rho))+\DIV(\rho\,b(\rho))=0,\quad \phi(\rho)=\f{\rho^2}{2\eta^{2}}+\int_0^{\rho}sF''(s)\diff s,\quad b(\rho)=\f{\nabla\omega_{\eta}\ast \rho}{\eta^{2}}.
\end{equation}

\begin{thm}\label{thm:class_sol_NLCH}
Equation~\eqref{eq:CHNL3} with initial condition $u_{0}$ satisfying~\eqref{eq:init_cond} admits a classical unique solution.
\end{thm}

To prove this theorem we first consider an approximate problem and we define $T_{\delta}$ a smooth function such that
\begin{align*}
T_{\delta}(0)=\f{\delta}{2}, \quad T_{\delta}(\rho)=\rho\text{ if ${\rho}\ge\delta$}, \quad T_{\delta}\text{ is increasing}. 
\end{align*}
The plan is to approximate \eqref{eq:CHNL1} with 
\begin{equation}\label{eq:C-H-mobility-apprx}
\p_{t}\rho=\DIV(T_{\delta}(\rho)\nabla\mu).    
\end{equation}
We also define
\begin{equation}\label{eq:phi_delta}
 \phi_{\delta}(\rho):=\int_{0}^\rho\f{T_{\delta}(s)}{\eta^{2}}\diff s+\int_{0}^{\rho}T_{\delta}(s)F''(s)\diff s=\int_{0}^{\rho}T_{\delta}(s)\left(\f{1}{\eta^{2}}+F''(s)\right)\diff s   
\end{equation}
so that equation~\eqref{eq:C-H-mobility-apprx} can be rewritten as a porous media equation
\begin{equation}\label{eq:CHNL3_approx}
\p_{t}\rho-\Delta(\phi_{\delta}(\rho))+\DIV(\rho\,b(\rho))=0\quad \rho(0,x)={\rho^0}(x).
\end{equation}
From the properties of $F$ {and for $\eta$ small enough} we note that $\phi_{\delta}\ge 0$ and $\phi_{\delta}'\ge 0$.  

\begin{lem}[existence]\label{lem:class_sol_approx}
There exists a classical solution to \eqref{eq:CHNL3_approx}. Moreover, the solution obeys the maximum principle 
\begin{equation*}
\underline{\rho}(t):= \sigma\exp\left(-\int_{0}^{t}\norm{\DIV b(\rho)}_{L^{\infty}}(s)\diff s\right) \leq \rho(t,x)\le{\max_{x\in\Td}\rho^{0}(x)}\exp\left(\int_{0}^{t}\norm{\DIV b(\rho)}_{L^{\infty}}(s)\diff s\right).
\end{equation*}
\end{lem}
\begin{rem}
We can apply the estimate from Lemma \ref{lem:class_sol_approx} in our case because 
$$
\|\DIV b(\rho(s,x))\|_{L^{\infty}} = \frac{1}{\eta^2} \|\Delta \omega_{\eta} \ast \rho(s,\cdot) \|_{L^{\infty}} \leq   \frac{1}{\eta^2}\, \|\Delta \omega_{\eta}\|_{L^{\infty}} \, \|\rho(s,\cdot)\|_{L^1} = C(\eta)\, \|\rho^0\|_{L^1} < \infty,
$$ 
as we deal with nonnegative solutions to \eqref{eq:CHNL3} and the mass is conserved.
\end{rem}
\begin{proof}[Proof of Lemma \ref{lem:class_sol_approx}]
The existence follows from~\cite{MR4574535}. To prove the maximum principle, we denote $w=\rho-\underline{\rho}$ so that
\begin{equation*}
\p_{t}w-\Delta(\phi_{\delta}(\rho))+\DIV(w\,b(\rho))+\underline{\rho}(\DIV(b(\rho))-\norm{\DIV(b(\rho))}_{L^{\infty}})=0, \quad w(0,x)={\rho^0}(x)-\underline{\rho}\ge 0.    
\end{equation*}
We multiply this equation by $\sgn^{-}(w):=\begin{cases}-1 \text{ if $w< 0$}\\ 0 \text{ if $w\ge 0$.}\end{cases}$. We obtain, with $w^{-}=\min\{w,0\}$, $|w^{-}|=-\min\{w,0\}$.
\begin{equation*}
\p_{t}|w^{-}|+\Delta(\phi_{\delta}(\rho))\sgn^{-}(w)+\DIV(|w^{-}|\,b(\rho))\le 0. 
\end{equation*}

Therefore integrating in space and using the inequality
\begin{equation*}
\int_{\Td}\Delta\phi_{\delta}(\rho)\sgn^{-}(w) { \, \diff x}\ge 0,
\end{equation*}
we obtain 
\begin{equation*}
\p_{t}\int_{\Td}|w^{-}| { \, \diff x} \le 0. 
\end{equation*}
Using the initial condition we conclude $|w^{-}|=0.$\\

Since the solutions to \eqref{eq:C-H-mobility-apprx} satisfy uniform lower bound, we obtain $T_{\delta}(\rho)=\rho$ for sufficiently small $\delta$ and thus classical solutions of Theorem~\ref{thm:class_sol_NLCH}.
\end{proof}

\begin{lem}[uniqueness]
Classical nonnegative solutions to \eqref{eq:CHNL3} are unique.
\end{lem}
\begin{proof}
We want to adapt usual $L^1$ contraction principle \cite[Proposition 3.5]{MR2286292} to the case with additional continuity equation term. Let $\rho_1$, $\rho_2$ be solutions to \eqref{eq:CHNL3_approx} and let $w = \rho_1 - \rho_2$. Equation for $w$ reads
$$
\partial_t w - \Delta (\phi(\rho_1) - \phi(\rho_2))+ \mbox{div}(\rho_1\,b(\rho_1) - \rho_2\,b(\rho_2)) = 0.
$$
We multiply this equation by $p_{\varepsilon}(\phi(\rho_1) - \phi(\rho_2))$ where $p_{\varepsilon}$ approximates $p(u) = \mathds{1}_{u > 0}$ and $p_{\varepsilon}' \geq 0$. Then,
$$
\int_{\Td} \Delta(\phi(\rho_1) - \phi(\rho_2)) \, p_{\varepsilon}(\phi(\rho_1) - \phi(\rho_2)) \diff x = - \int_{\Td} p_{\varepsilon}' \, |\nabla(\phi(\rho_1) - \phi(\rho_2))|^2 \diff x \leq 0.
$$
Concerning the other terms we notice that after sending $\varepsilon \to 0$ we obtain $p(\phi(\rho_1) - \phi(\rho_2)) = p(\rho_1 - \rho_2)$ by monotonicity of $\phi$. Therefore,
$$
\int_{\Td} \partial_t w \, p(\rho_1 - \rho_2) \diff x = \partial_t \int_{\Td} |w|_+ \diff x.
$$
Now, we split the divergence term into two parts:
$$
\mbox{div}(\rho_1\,b(\rho_1) - \rho_2\,b(\rho_2)) = \left[\rho_1\, \mbox{div}b(\rho_1) - \rho_2\, \mbox{div}b(\rho_2)\right]  + \left[\nabla \rho_1 \, b(\rho_1) - \nabla \rho_2 \, b(\rho_2)\right] = A+B.
$$
The term $A$ can be estimated in $L^1(\Td)$ with
\begin{align*}
\|A\|_{1} &\leq \|\rho_1\, \mbox{div}b(\rho_1) - \rho_2\, \mbox{div}b(\rho_1)\|_{1} +  \|\rho_2\, \mbox{div}b(\rho_1) -  \rho_2\, \mbox{div}b(\rho_2)\|_{1}\\
&\leq \|\rho_1 - \rho_2\|_{1} \, \|\mbox{div}b(\rho_1)\|_{\infty} + \frac{1}{\eta^2} \|\rho_2\|_{1} \, \|D^2\omega_{\eta}\|_{\infty} \|\rho_1 - \rho_2\|_{1}\\
&\leq \frac{\|D^2\omega_{\eta}\|_{\infty}}{\eta^2}(\|\rho_1\|_{1} + \|\rho_2\|_{1}) \, \|\rho_1 - \rho_2\|_{1} 
\end{align*}
where we used Young's convolutional inequality. Therefore, 
$$
\int_{\Td} p \, A \diff x \leq \|p\, A \|_{1} \leq \frac{\|D^2\omega_{\eta}\|_{\infty}}{\eta^2}(\|\rho_1\|_{1} + \|\rho_2\|_{1}) \, \|\rho_1 - 
\rho_2\|_{1}.
$$
where we denoted for simplicity $p = p(\rho_1 - \rho_2)$. Concerning term $B$ we write similarly
$$
B  = \left(\nabla \rho_1 \, b(\rho_1) -\nabla \rho_2 \, b(\rho_1)\right) + \left(\nabla \rho_2 \, b(\rho_1) - \nabla \rho_2 \, b(\rho_2)\right) =: B_1 + B_2.
$$
As above, we easily obtain
$$
\|B_2\|_{1} \leq \frac{\|\nabla\omega_{\eta}\|_{\infty}}{\eta^2} \, \|\nabla \rho_2\|_{1}\, \|\rho_1 - \rho_2\|_{1}, \quad \int_{\Td} p\,B_2 \diff x \leq \frac{\|\nabla\omega_{\eta}\|_{\infty}}{\eta^2} \, \|\nabla \rho_2\|_{1}\, \|\rho_1 - \rho_2\|_{1}.
$$
The term $B_1$ is more tricky. Keeping in mind that everything is multiplied by $p(\rho_1-\rho_2)$ we have
\begin{multline*}
\int_{\Td} \left(\nabla \rho_1 -\nabla \rho_2 \right) \, p(\rho_1 - \rho_2)  \, b(\rho_1) \diff x = \int_{\Td} \nabla |\rho_1 - \rho_2|_+ \, b(\rho_1) \diff x =\\= - \int_{\Td} |\rho_1 - \rho_2|_+ \mbox{div} b(\rho_1) \diff x 
\leq \|\rho_1 - \rho_2\|_{1} \, \|\rho_1\|_{1} \,\frac{\|D^2\omega_{\eta}\|_{\infty}}{\eta^2}.
\end{multline*}
We conclude that for some constant $C$ depending on $L^1$ norms of $\rho_1$, $\rho_2$ and $\nabla \rho_2$ we have
$$
\partial_t \int_{\Td} |\rho_1 - \rho_2|_+ \diff x \leq {C}\int_{\Td} |\rho_1 - \rho_2| \diff x.
$$
Replacing $\rho_1$ and $\rho_2$ we obtain
$$
\partial_t \int_{\Td} |\rho_1 - \rho_2| \diff x \leq {C}\int_{\Td} |\rho_1 - \rho_2| \diff x.
$$
so that we conclude $\rho_1 = 
\rho_2$.
\end{proof}

\section{Convergence of nonlocal Euler-Korteweg to nonlocal Cahn-Hilliard}\label{sect:conv_rel_entr}

To prove convergence of nonlocal Euler-Korteweg equation to the nonlocal Cahn-Hilliard equation, we first rewrite the latter as a nonlocal Euler-Korteweg equation with an error term:
\begin{align}
&\p_{t}\Rho+\f{1}{\eps}\DIV(\Rho\Bu)=0,\label{eq:EKCH1}\\
&\p_{t}(\Rho\Bu)+\f{1}{\eps}\DIV\left(\Rho\Bu\otimes\Bu\right)=-\f{1}{\eps^{2}} \Rho\Bu-\f{1}{\eps}\Rho\nabla(F'(\Rho)+B_{\eta}[\Rho])+e(\Rho,\Bu).\label{eq:EKCH2}
\end{align}
%eq:EKCH1 eq:defBu
Here, velocity $\Bu$ is given by
\begin{equation}\label{eq:defBu}
\Bu=-\eps\nabla(F'(\Rho)-B_{\eta}(\Rho))     
\end{equation}
and $\Rho$ is the solution of the nonlocal Cahn-Hilliard equation. The error term is given by 
\begin{align*}
e(\Rho,\Bu)&=\p_{t}(\Rho\Bu)+\f{1}{\eps}\DIV\left(\Rho\Bu\otimes\Bu\right)\\
&=\eps\DIV(\Rho\nabla(F'(\Rho)+B_{\eta}[\Rho]))\otimes\nabla(F'(\Rho)+B_{\eta}[\Rho])))-\eps\p_{t}(\Rho\nabla(F'(\Rho)+B_{\eta}[\Rho]))).     
\end{align*}
Finally, given strong solution $(P,\Bu)$ and measure-valued solution represented by $(\rho, \overline{\sqrt{\rho}\bu}, \nu,m)$ we define the relative entropy as 
\begin{equation}\label{eq:relative_entropy}
\Theta(t)=\int_{\Td}\f{1}{2}\overline{\rho|\bu-\Bu|^{2}}+\overline{F(\rho|\Rho)}\diff x+\f{1}{4\eta^{2}}\int_{\Td}\overline{\int_{\Td}\omega_{\eta}(y)|(\rho-\Rho)(x)-(\rho-\Rho)(x-y)|^{2}\diff y} \diff x.     
\end{equation}
where nonlinearity $F(\rho|\Rho)$ is defined in \eqref{eq:F(rho-Rho)} and measure-valued terms are defined by \eqref{eq:def_rel_kin_energy}, \eqref{eq:nonlocal_term_with_P} and \eqref{eq:def_rel_potenial_rho_Rho}. The main result reads:
\begin{thm}\label{thm:rel_entr_estimate}
Let $(\rho, \overline{\sqrt{\rho}\bu}, \nu, m)$ be a dissipative measure valued solution of \eqref{eq:EK1}--\eqref{eq:EK2} satisfying Poincaré inequality \eqref{eq:nonloc_Poinc_stronger_bar} and let $(\Rho,\Bu)$ be classical solutions of \eqref{eq:EKCH1}--\eqref{eq:EKCH2}. Then, for a constant independent of $\varepsilon$ and $\eta$ we have
\begin{equation}\label{est:relative_entropy_4}
\Theta(t)\le \left(\Theta(0)+\eps^{4} C(\|\Rho\|_{C^{2,1}}) \, \left\| \frac{1}{\Rho}  \right\|_{\infty}^2 \right) \, e^{TC(\|\Rho\|_{C^{2,1}})/\eta^{d+3}}.   
\end{equation}
\end{thm}
\begin{lem}\label{lem:rel_entr_nonneg}
Let $\eta \in (0,\eta_0)$. Then, the relative entropy defined by \eqref{eq:relative_entropy} is nonnegative: there exists a $\kappa \in (0,1)$ such that
\begin{equation}\label{eq:nonnegativity_rel_entropy_explicit}
\int_{\Td} \overline{\rho|\bu-\Bu|^{2}} \diff x  + \f{\kappa}{4\eta^{2}}\int_{\Td}\overline{\int_{\Td}\omega_{\eta}(y)|(\rho-\Rho)(x)-(\rho-\Rho)(x-y)|^{2}\diff y} \diff x \leq \Theta(t)
\end{equation}
where both terms on the (LHS) are nonnegative. Moreover, for the constant $C_P$ (defined in Lemma \ref{lem:Poincaré_nonlocal_H1_L2}) we have an estimate
\begin{equation}\label{eq:rho-Rho_via_Ponce_bar}
\|\rho - \Rho\|^2_{L^2(\Td)} \leq \f{C_P}{4\eta^{2}}\int_{\Td}\overline{\int_{\Td}\omega_{\eta}(y)|(\rho-\Rho)(x)-(\rho-\Rho)(x-y)|^{2}\diff y} \diff x.
\end{equation}
\end{lem}
\begin{proof}[Proof of Lemma \ref{lem:rel_entr_nonneg}]
We study the three terms appearing in \eqref{eq:relative_entropy} separately. First, for \eqref{eq:def_rel_kin_energy} we write by Fubini theorem
$$
\overline{\rho|\bu-\Bu|^{2}} = \left \langle |\lambda'|^2 + \lambda_1 \, |\Bu|^2 - 2 \sqrt{\lambda_1}\,\lambda' \Bu, \nu_{t,x} \right \rangle + m^{\rho|\bu|^2} = \left \langle |\lambda' - \sqrt{\lambda_1}\Bu |^2, \nu_{t,x} \right \rangle + m^{\rho|\bu|^2},
$$
so that, after integration in space, it is positive (for a.e. $t\in (0,T)$). Now, we study the nonlocal term. We claim that (after integration)
\begin{equation}\label{eq:est_conv_sq_sq_conv_Pr}
\int_{\Td} \overline{\int_{\Td} \omega_{\eta}(y)|\rho(x) - \rho(x-y)|^2 \diff y} \diff x \geq \int_{\Td} \int_{\Td} \omega_{\eta}(y)|\rho(x) - \rho(x-y)|^2 \diff y \diff x.
\end{equation}
Indeed, by definition \eqref{eq:def_convolution_rho_squared}, the (LHS) equals $2 \int_{\Td} \overline{\rho^{2}} {\, \diff x}- \int_{\Td} 2\,\rho \, \omega_{\eta}\ast\rho{\, \diff x}$. By \eqref{eq:estimate_rho_sq_sq_rho}, we know that $\int_{\Td} \overline{\rho^{2}}{\, \diff x} \geq \int_{\Td} \rho^2{\, \diff x}$. To conclude the proof of \eqref{eq:est_conv_sq_sq_conv_Pr}, it is sufficient to observe
$$
2 \int_{\Td} {\rho^{2}}{\, \diff x} - \int_{\Td} 2\,\rho \, \omega_{\eta}\ast\rho{\, \diff x} = \int_{\Td} \int_{\Td} \omega_{\eta} |\rho(x) - \rho(x-y)|^2 \diff y \diff x.
$$
Now, combining \eqref{eq:nonlocal_term_with_P} and \eqref{eq:est_conv_sq_sq_conv_Pr}, we obtain
\begin{multline*}
\int_{\Td}\overline{\int_{\Td}\omega_{\eta}(y)|(\rho-\Rho)(x)-(\rho-\Rho)(x-y)|^{2}\diff y} \diff x \geq \\ \geq \int_{\Td}{\int_{\Td}\omega_{\eta}(y)|(\rho-\Rho)(x)-(\rho-\Rho)(x-y)|^{2}\diff y} \diff x.
\end{multline*}
Using Lemma \ref{lem:Poincaré_nonlocal_H1_L2}, we conclude the proof of \eqref{eq:rho-Rho_via_Ponce_bar} and nonnegativity of the nonlocal term.\\

It remains to study the term $\overline{F(\rho|\Rho)}$. The concentration measure $m^{F(\rho)}$ is nonnegative and will be neglected in the estimate below. We split $F= F_1 + F_2$ (where $F_1$, $F_2$ are defined in Assumption~\ref{ass:potentialF}) in \eqref{eq:def_rel_potenial_rho_Rho} so that from \eqref{eq:id_wl_8} and \eqref{eq:def_rel_potenial_rho_Rho}
\begin{equation}\label{eq:F_rel_pot_splitted_for_two_parts_mvs}
\begin{split}
\overline{F(\rho|\Rho)} = & \left \langle F_1(\lambda_1) - F_1(\Rho) - F_1'(\Rho) (\lambda_1 - \Rho), \nu_{t,x} \right \rangle
\\ &+ \left \langle F_2(\lambda_1) - F_2(\Rho) - F_2'(\Rho) (\lambda_1 - \Rho), \nu_{t,x} \right \rangle + m^{F(\rho)}
\end{split} 
\end{equation}
The first term is nonnegative by convexity of $F_1$. The second can be estimated from below (by Taylor's expansion) with $- \|F_2''\|_{\infty} \left \langle (\lambda_1 - \Rho)^2, \nu_{t,x} \right \rangle$. Now, recall that $\|F_2''\|_{\infty} < \frac{1}{C_P}$ (cf. Assumption~\ref{ass:potentialF}). In particular, there exists $\kappa \in (0,1)$ such that $\|F_2''\|_{\infty} < \frac{1-\kappa}{C_P}$. Using Poincaré inequality \eqref{eq:nonloc_Poinc_stronger_bar} and the fact that the concentration measure $m^{\rho^2}$ is nonnegative we have
\begin{equation}\label{eq:F_rel_pot_splitted_for_two_parts_mvs_2}
\begin{split}
- \|F_2''\|_{\infty} \int_{\Td} \left \langle (\lambda_1 - \Rho)^2, \nu_{t,x} \right \rangle \diff x  \geq 
- \|F_2''\|_{\infty} \int_{\Td} \overline{(\rho - \Rho)^2} \diff x \geq - \frac{{(1-\kappa)}}{C_P} \int_{\Td} \overline{(\rho - \Rho)^2} \diff x \\
\geq  -\frac{1-\kappa}{4\eta^2}\int_{\Td}\overline{\int_{\Td}\omega_{\eta}(y)|(\rho-\Rho)(x)-(\rho-\Rho)(x-y)|^{2}\diff y} \diff x.
\end{split} 
\end{equation}
Therefore, we can compensate a possibly negative term with the positive nonlocal term appearing in \eqref{eq:relative_entropy}.
 \end{proof}

\begin{proof}[Proof of Theorem \ref{thm:rel_entr_estimate}]
We split the reasoning into several steps.

\underline{\it Step 1: Energy identities.} First, we recall that the dissipative measure valued solutions satisfy 
\begin{equation}\label{dissip_energy_meas}
\begin{aligned}
\int_{\Td}\f{1}{2}\overline{\rho|\bu|^{2}}+\overline{F(\rho)}\diff x+\f{1}{4\eta^{2}} \int_{\Td}\overline{\int_{\Td}\omega_{\eta}(y)|\rho(x)-\rho(x-y)|^{2}\diff y}\diff x   +\f{1}{\eps^{2}}\int_{0}^{t}\int_{\Td}\overline{\rho|\bu|^{2}}\diff x\diff t\\ \leq \int_{\Td}\frac{1}{2}{\rho^0}|u_0|^2(x)+F({\rho^0})\diff x+\f{1}{4\eta^{2}}\int_{\Td}\int_{\Td}\omega_{\eta}(y)|{\rho^0}(x)-{\rho^0}(x-y)|^{2}\diff y\diff x.  
\end{aligned}
\end{equation}
where the quantities on the (LHS) of \eqref{dissip_energy_meas} are evaluated at time $t$. Similarly, the classical solutions $(\Rho,\Bu)$ satisfy 
\begin{equation}\label{dissip_energy_class}
\begin{aligned}
&\int_{\Td}\f{1}{2}\Rho|\Bu|^{2}+F(\Rho)\diff x+\f{1}{4\eta^{2}}\int_{\Td}\int_{\Td}\omega_{\eta}(y)|\Rho(x)-\Rho(x-y)|^{2}\diff x\diff y = \\ 
& =   \int_{\Td}\frac{1}{2}{\Rho^0}|u_0|^2(x)+F({\Rho^0})\diff x+\f{1}{4\eta^{2}}\int_{\Td}\int_{\Td}\omega_{\eta}(y)|{\Rho^0}(x)-{\Rho^0}(x-y)|^{2}\diff x\diff y\\
& \phantom{ =} - \f{1}{\eps^{2}}\int_{0}^{t}\int_{\Td}\Rho|\Bu|^{2}\diff x\diff t  + \int_{0}^{t}\int_{\Td}\Bu\cdot e(\Rho,\Bu)\diff t\diff x.
\end{aligned}
\end{equation}
Identity \eqref{dissip_energy_class} can be obtained from testing \eqref{eq:EKCH1}--\eqref{eq:EKCH2} by $\Bu$ and performing several integration by parts.\\

\underline{\it Step 2: Estimate for the mixed terms $F'(\Rho)({\rho}-\Rho)$, $B_{\eta}[\Rho]$ and ${\rho}\,|\Bu|^2$.} 
We consider weak solutions of the mass equation satisfied by the differences between the measure valued solutions and classical solutions: 
\begin{equation}\label{eq:weak_form_mass_cons}
\int_0^T \int_{\Td}\partial_t\psi ({\rho}-\Rho)+\f{1}{\eps}\nabla\psi\cdot(\overline{\rho\bu}-\Rho\Bu) \diff x \diff t + \int_{\Td}\psi(x,0)({\rho^0}-{\Rho^0})\diff x=0.
\end{equation}
We set 
\begin{equation*}
\theta_{\delta}(t):=
\begin{cases}1\quad &\text{for $0\le\tau<t$},\\
\f{t-\tau}{\delta}+1\quad &\text{for $t\le\tau<t+\delta$},\\
0\quad&\text{for $\tau\ge t+\delta$.}
\end{cases}
\end{equation*}
Note that $\theta'(t)$ is an approximation of the dirac mass $-\delta_{t}$. We consider test function in \eqref{eq:weak_form_mass_cons} defined as $\psi=\theta_{\delta}(t)\left(F'(\Rho)+B_{\eta}[\Rho]-\f{1}{2}|\Bu|^{2}\right)$ so that after letting $\delta \to 0$ we obtain
\begin{equation}\label{eq:difference_energy_mass}
\begin{aligned}
&\int_{\Td}\left(F'(\Rho)+B_{\eta}[\Rho]-\f{1}{2}|\Bu|^{2}\right)({\rho}-\Rho)\Big|_{\tau=0}^{t}\diff x
= \\ 
&= +\int_{0}^{t}\int_{\Td}\p_{\tau}\left(F'(\Rho)+B_{\eta}[\Rho]-\f{1}{2}|\Bu|^{2}\right)({\rho}-\Rho)\diff x\diff\tau \\
&\phantom{=  }+\f{1}{\eps}\int_{0}^{t}\int_{\Td}\nabla\left(F'(\Rho)+B_{\eta}[\Rho]-\f{1}{2}|\Bu|^{2}\right)\cdot(\overline{\rho\bu}-\Rho\Bu)\diff x\diff\tau. 
\end{aligned}    
\end{equation}
\underline{\it Step 3: Estimate for the mixed term $\overline{\rho \bu}\,\Bu$.} We consider weak solutions of the momentum equation satisfied by the differences between the measure valued solutions and classical solutions: 
\begin{multline*}
\int_0^\infty\int_{\Td}\partial_t\phi\cdot(\overline{\rho \bu}-\Rho\Bu)+\f{1}{\eps}\nabla\phi : (\overline{\rho \bu\otimes \bu}-\Rho\Bu\otimes\Bu)-\f{1}{\eps^{2}}\phi\cdot(\overline{\rho\bu}-\Rho\Bu)+\f{1}{\eps}\DIV\phi(\overline{p(\rho)}-p(\Rho)) {\,\diff x\diff t}\\+\f{1}{\eps\eta^{2}}\phi\cdot({\rho \nabla\omega_{\eta}\ast\rho}-\Rho\cdot\nabla\omega_{\eta}\ast\Rho){\diff x\diff t}
+\int_{\Td}\phi(x,0)\cdot ({\rho^0}{ \bu^0}-{\Rho^0}{ \Bu^0})\diff x-\int_{0}^{\infty}\int_{\Td}\phi\cdot e(\Rho,\Bu){\diff x\diff t}=0.
\end{multline*}
We consider test function $\phi=\theta_{\delta}(t)\Bu$ so that after letting $\delta\to 0$ we obtain 
\begin{equation}\label{eq:difference_energy_momentum}
\begin{aligned}
&\int_{\Td}\Bu\cdot\left(\overline{\rho\bu}-\Rho\Bu\right)\Big|_{\tau=0}^{t}\diff x = \int_{0}^{t}\int_{\Td}\p_{\tau}\Bu \cdot (\overline{\rho\bu}-\Rho\Bu)\diff x\diff\tau\\
&+\int_{0}^{t}\int_{\Td}\f{1}{\eps}\nabla\Bu:(\overline{\rho \bu\otimes \bu}-\Rho\Bu\otimes\Bu)+\f{1}{\eps}\DIV(\Bu)(\overline{p(\rho)}-p(\Rho))\diff x\diff \tau\\
&-\f{1}{\eps^{2}}\int_{0}^{t}\int_{\Td}\Bu \cdot (\overline{\rho\bu}-\Rho\Bu)\diff x\diff \tau+\f{1}{\eps\eta^{2}}\int_{0}^{t}\int_{\Td}\Bu\cdot({\rho \nabla\omega_{\eta}\ast\rho}-\Rho\cdot\nabla\omega_{\eta}\ast\Rho)\diff x\diff\tau
\\&-\int_{0}^{t}\int_{\Td}\Bu\cdot e(\Rho,\Bu)\diff x\diff \tau.
\end{aligned}    
\end{equation}
\underline{\it Step 4: First estimate on the relative entropy.} Let us observe that when we subtract \eqref{dissip_energy_class}, \eqref{eq:difference_energy_mass} and \eqref{eq:difference_energy_momentum} from \eqref{dissip_energy_meas} we obtain an estimate for $\Theta(t)-\Theta(0)$. To see this, let us write explicitly the (LHS) after the subtraction (we omit integral with respect to $x$ for simplicity and we consider only terms at time $\tau = t$; of course, for $\tau = 0$, they will be analogous):
\begin{multline*}
\f{1}{2}\overline{\rho|\bu|^{2}}+\overline{F(\rho)} +\f{1}{4\eta^{2}}\overline{\int_{\Td}\omega_{\eta}(y)|\rho(x)-\rho(x-y)|^{2}\diff y} - \f{1}{2}\Rho|\Bu|^{2} - F(\Rho)\\
 - \f{1}{4\eta^{2}}\int_{\Td}\omega_{\eta}(y)|\Rho(x)-\Rho(x-y)|^{2}\diff y 
-\left(F'(\Rho)+B_{\eta}[\Rho]-\f{1}{2}|\Bu|^{2}\right)({\rho}-\Rho) - \Bu\cdot\left(\overline{\rho\bu}-\Rho\Bu\right).
\end{multline*}
We claim that this expression equals $\Theta(t)$. Indeed, the terms containing both density and velocity sum up to the term $\overline{\rho|\bu-\Bu|^{2}}$ as in \eqref{eq:def_rel_kin_energy}. Similarly, terms with the potential $F$ and its derivative $F'$ can be combined to \eqref{eq:def_rel_potenial_rho_Rho}. Finally, for the nonlocal term, the claim is the consequence of two identities:
$$
B_{\eta}[\Rho]\, {\rho} = \f{1}{2\eta^{2}}\int_{\Td}{\omega_{\eta}(y) (\Rho(x)-\Rho(x-y))} \, ({\rho}(x)-{\rho}(x-y))\diff y
$$
and the similar one for $B_{\eta}[\Rho] \, \Rho$. Hence, the expression of interest equals $\Theta(t)$. Subtracting all the terms on the (RHS) of \eqref{dissip_energy_class}, \eqref{eq:difference_energy_mass},\eqref{eq:difference_energy_momentum} from (RHS) of \eqref{dissip_energy_meas} we obtain
\begin{equation}\label{est:relative_entropy_1}
\begin{aligned}
&\Theta(t)-\Theta(0)\le -\f{1}{\eps^{2}}\int_{0}^{t}\int_{\Td}\overline{\rho|\bu|^{2}}-\Rho|\Bu|^{2}-\Bu \cdot (\overline{\rho\bu}-\Rho\Bu)\diff x\diff \tau\\
& \qquad \qquad-\int_{0}^{t}\int_{\Td}\p_{\tau}\left(F'(\Rho)+B_{\eta}[\Rho]-\f{1}{2}|\Bu|^{2}\right)({\rho}-\Rho)+\p_{\tau}\Bu \cdot (\overline{\rho\bu}-\Rho\Bu)\diff x\diff\tau\\
& \qquad \qquad-\f{1}{\eps}\int_{0}^{t}\int_{\Td}\nabla\left(F'(\Rho)+B_{\eta}[\Rho]-\f{1}{2}|\Bu|^{2}\right)\cdot(\overline{\rho\bu}-\Rho\Bu)\diff x\diff\tau\\
& \qquad \qquad -\f{1}{\eps}\int_{0}^{t}\int_{\Td}\nabla\Bu:(\overline{\rho \bu\otimes \bu}-\Rho\Bu\otimes\Bu)+\DIV(\Bu)(\overline{p(\rho)}-p(\Rho))\diff x\diff \tau\\
& \qquad \qquad -\f{1}{\eps}\f{1}{\eta^{2}}\int_{0}^{t}\int_{\Td}\Bu\cdot({\rho \nabla\omega_{\eta}\ast\rho}-\Rho\cdot\nabla\omega_{\eta}\ast\Rho)\diff x\diff\tau. 
%=: A+B+C+D+E+F. 
\end{aligned}   
\end{equation}
\underline{\it Step 5: Terms with $\partial_\tau \Bu$ in \eqref{est:relative_entropy_1}.}
To estimate the right-hand side of \eqref{est:relative_entropy_1} we first try to eliminate time derivative from \eqref{est:relative_entropy_1}. To this end, we compute $\p_{t}\Bu$ from the equations~\eqref{eq:EKCH1}-\eqref{eq:EKCH2} to obtain that $\Bu$ satisfies
\begin{equation}\label{eq:PDEBu}
\p_{t}\Bu+\f{1}{\eps}\left(\Bu\cdot\nabla\right)\Bu=-\f{1}{\eps^{2}} \Bu-\f{1}{\eps}\nabla(F'(\Rho)+B_{\eta}[\Rho])+\f{e(\Rho,\Bu)}{\Rho}.     
\end{equation}
We take the scalar product of this equation with $\overline{\rho\,\bu} - {\rho}\,\Bu$ which yields 
\begin{equation*}
\begin{aligned}
&\partial_t\Bu \cdot (\overline{\rho\bu} - \Rho \Bu ) + \f{1}{2}\partial_t|\Bu|^2 (\Rho - {\rho}) +\f{1}{\eps}\nabla \Bu : (\overline{\rho\bu}\otimes\Bu-{\rho}\Bu\otimes\Bu)\\
& \qquad \qquad =\f{1}{\eps^{2}}({\rho}|\Bu|^{2}-\Bu\cdot\overline{\rho\bu})-\f{1}{\eps}\nabla(F'(\Rho)+B_{\eta}[\Rho]) \cdot (\overline{\rho\bu}-{\rho}\Bu)+\f{e(\Rho,\Bu)}{\Rho} \cdot (\overline{\rho\bu}-{\rho}\Bu),     
\end{aligned}
\end{equation*}
where we used identities
$$
\f{1}{\eps} (\Bu \cdot \nabla) \Bu \cdot (\overline{\rho\,\bu} - {\rho} \Bu)= \f{1}{\eps}\nabla \Bu : (\overline{\rho\bu}\otimes\Bu-{\rho}\Bu\otimes\Bu),
$$
$$
\partial_t \Bu \cdot (\overline{\rho\bu} - {\rho}\Bu) = \partial_t\Bu \cdot (\overline{\rho\bu} - \Rho \Bu) + \frac{1}{2}\partial_t |\Bu|^2 (\Rho - {\rho}).
$$
Finally, using matrix identity $x\, A\, y = A : x\otimes y$ where $x, y \in \R^d$ and $A \in \R^{d \times d}$ we easily deduce the formula
\begin{multline*}
\nabla \Bu : (\overline{\rho\bu}\otimes\Bu- {\rho}\Bu\otimes\Bu) = \nabla\Bu :(\overline{\rho\bu \otimes \bu} - \Rho \Bu \otimes \Bu)
\\
- \nabla\Bu:\overline{\rho(\bu-\Bu)\otimes(\bu-\Bu)} - \nabla\left(\frac{1}{2}|\Bu|^2\right)(\overline{\rho\bu}-\Rho\Bu),
\end{multline*}
where
$$
\overline{\rho(\bu-\Bu)\otimes(\bu-\Bu)}:= \overline{\rho \bu \otimes \bu} - \overline{\rho \bu} \otimes \Bu - \Bu \otimes \overline{\rho \bu} + \rho \Bu \otimes \Bu.
$$
We obtain
\begin{equation}\label{eq:sum_mes_mass_moment}
\begin{aligned}
&\f{1}{2}\p_{t} |\Bu|^{2}\,(\Rho-{\rho})+\p_{t}(\Bu) \cdot (\overline{\rho\bu}-\Rho\Bu)-\f{1}{\eps}\nabla\left(\f{1}{2}|\Bu|^{2}\right)\cdot(\overline{\rho\bu}-\Rho\Bu)+\\
&+\f{1}{\eps}\nabla \Bu : (\overline{\rho\bu\otimes\bu}
 -\Rho\Bu\otimes\Bu) =\f{1}{\eps}\nabla\Bu:\overline{\rho(\bu-\Bu)\otimes(\bu-\Bu)} +\\
& {+}\f{1}{\eps^{2}}({\rho}|\Bu|^{2}-\Bu\cdot\overline{\rho\bu}) -\f{1}{\eps}\nabla(F'(\Rho)+B_{\eta}[\Rho]) \cdot (\overline{\rho\bu}-{\rho}\Bu)
+\f{e(\Rho,\Bu)}{\Rho} \cdot (\overline{\rho\bu}-{\rho}\Bu).     
\end{aligned}
\end{equation}
Note that this gives us an estimate on four terms appearing on the (RHS) of \eqref{est:relative_entropy_1}.\\

\noindent \underline{\it Step 6: Terms with $F'$ and $B_{\eta}$ in \eqref{est:relative_entropy_1}.} We now consider the expression
\begin{multline*}
-\int_{0}^{t}\int_{\Td}\p_{\tau}\left(F'(\Rho)+B_{\eta}[\Rho]\right)({\rho}-\Rho)
+\f{1}{\eps} \nabla\left(F'(\Rho)+B_{\eta}[\Rho]\right)\cdot(\overline{\rho\bu}-\Rho\Bu)\diff x\diff\tau \\
 + \int_{0}^{t}\int_{\Td} \f{1}{\eps}\nabla(F'(\Rho)+B_{\eta}[\Rho]) \cdot (\overline{\rho\bu}-{\rho}\Bu) \diff x\diff\tau.
\end{multline*}
The first integral comes from \eqref{est:relative_entropy_1} while the second from \eqref{eq:sum_mes_mass_moment} plugged into \eqref{est:relative_entropy_1}. We can simplify this to get
\begin{equation}\label{eq:nonlocal_term_summed}
-\int_{0}^{t}\int_{\Td}\p_{\tau}\left(F'(\Rho)+B_{\eta}[\Rho]\right)({\rho}-\Rho)
+\f{1}{\eps} \nabla\left(F'(\Rho)+B_{\eta}[\Rho]\right)\cdot \Bu ({\rho}-\Rho)\diff x\diff\tau.
\end{equation}
We split the term with $B_{\eta}[\Rho] = \frac{\Rho}{\eta^2} - \frac{\Rho\ast\omega_{\eta}}{\eta^2}$ for the local and non-local parts. Now, concerning the terms with potential $F$, we use \eqref{eq:EKCH1} to deduce
$$
\partial_t F'(\Rho) = F''(\Rho)\, \partial_t \Rho = 
-\frac{1}{\varepsilon} F''(\Rho) \, \nabla \Rho \cdot \Bu 
-\frac{1}{\varepsilon} F''(\Rho) \, \Rho \, \mbox{div} \Bu 
=-\frac{1}{\varepsilon} \nabla F'(\Rho) \cdot \Bu -\frac{1}{\varepsilon} F''(\Rho) \, \Rho \, \mbox{div} \Bu.
$$
Similarly,
$$
\frac{1}{\eta^2}\partial_t \Rho = -\frac{1}{\varepsilon\,\eta^2} \nabla \Rho \cdot \Bu -\frac{1}{\varepsilon\,\eta^2}  \Rho \, \mbox{div} \Bu.
$$
Therefore, the local parts of \eqref{eq:nonlocal_term_summed} sum up to
$$
-\frac{1}{\varepsilon} \left( F''(\Rho) \, \Rho + \frac{1}{\eta^2}\,\Rho \right)\, \mbox{div}\Bu\,({\rho} - \Rho) = -\frac{1}{\varepsilon}p'(\Rho)  \, \mbox{div}\Bu\, ({\rho} - \Rho)
$$
which together with $-\frac{1}{\varepsilon}\DIV(\Bu)(\overline{p(\rho)}-p(\Rho))\, \mbox{div}\Bu$ from \eqref{est:relative_entropy_1} gives $\overline{p(\rho|\Rho)}\,\mbox{div}\Bu$,
where
$$
\overline{p(\rho|\Rho)} := \overline{p(\rho)} - p(\Rho) - p'(\Rho)\,(\rho-\Rho).
$$

Now, we consider the nonlocal parts in \eqref{eq:nonlocal_term_summed} and the last nonlocal term
coming from \eqref{est:relative_entropy_1}. which equals
\begin{equation}\label{eq:nonlocal_term_relative_entropy_RHS}
\begin{split}
\frac{1}{\eta^2}\int_{0}^{t}\int_{\Td}\p_{\tau}\left(\Rho\ast\omega_{\eta}\right)({\rho}-\Rho)
+\f{1}{\eps\,\eta^2} &\nabla\left(\Rho\ast\omega_{\eta}\right)\cdot \Bu ({\rho}-\Rho)\diff x\diff\tau \\
&-\f{1}{\eps}\f{1}{\eta^{2}}\int_{0}^{t}\int_{\Td}\Bu\cdot({\rho \nabla\omega_{\eta}\ast\rho}-\Rho\cdot\nabla\omega_{\eta}\ast\Rho)\diff x\diff\tau. 
\end{split}
\end{equation}
Using~\eqref{eq:EKCH1} and properties of the convolution we can rewrite the first term in \eqref{eq:nonlocal_term_relative_entropy_RHS}:
\begin{equation*}
\begin{aligned}
\f{1}{\eta^{2}} \int_{0}^{t}\int_{\Td}\p_{\tau}\left(\Rho \ast \omega_{\eta}\right)({\rho}-\Rho)\diff x\diff\tau&=-\f{1}{\eps}\f{1}{\eta^{2}}\int_{0}^{t}\int_{\Td}\DIV(\Rho\Bu)\left(\omega_{\eta}\ast({\rho}-\Rho)\right)\diff x\diff\tau\\
&=\f{1}{\eps}\f{1}{\eta^{2}}\int_{0}^{t}\int_{\Td}\Rho\Bu\cdot\left(\nabla\omega_{\eta}\ast({\rho}-\Rho)\right)\diff x\diff\tau.
\end{aligned}    
\end{equation*}
so that \eqref{eq:nonlocal_term_relative_entropy_RHS} boils down to 
$$
-\f{1}{\eps}\f{1}{\eta^{2}}\int_{0}^{t}\int_{\Td}({\rho}-\Rho) \Bu\cdot\nabla\omega_{\eta}\ast\left({\rho}-\Rho\right)\diff x\diff\tau.
$$
\noindent \underline{\it Step 7: Final estimate on the relative entropy.}
Using the steps above and \eqref{est:relative_entropy_1} we obtain
\begin{equation}\label{est:relative_entropy_3}
\begin{aligned}
\Theta(t)-\Theta(0)&\le \underbrace{-\f{1}{\eps^{2}}\int_{0}^{t}\int_{\Td}\overline{\rho|\bu-\Bu|^{2}}\diff x\diff \tau}_{{=:A}} \, \underbrace{-\int_{0}^{t}\int_{\Td}\f{e(\Rho,\Bu)}{\Rho}(\overline{\rho\bu} -{\rho}\Bu)\diff x\diff \tau}_{{=:B}}\\
&\underbrace{-\f{1}{\eps}\int_{0}^{t}\int_{\Td}\nabla\Bu:\overline{\rho(\bu-\Bu)\otimes(\bu-\Bu)}\diff x\diff \tau}_{{=:C}} \, \underbrace{-\f{1}{\eps}\int_{0}^{t}\int_{\Td}\DIV(\Bu)\,\overline{p({\rho}|\Rho)}\diff x\diff \tau}_{{=:D}} \\
&\underbrace{-\f{1}{\eps}\f{1}{\eta^{2}}\int_{0}^{t}\int_{\Td}({\rho}-\Rho) \Bu\cdot\nabla\omega_{\eta}\ast\left({\rho}-\Rho\right)\diff x\diff\tau.}_{{=:E}}
\end{aligned}   
\end{equation}
By definition of $\Bu$ we notice that 
$$
\norm{\Bu}_{\infty}, \norm{\nabla \Bu}_{\infty},  |e(\Rho, \Bu) | \leq \varepsilon\,C(\|\Rho\|_{C^{2,1}}),
$$
where $C(\|\Rho\|_{C^{2,1}})$ is a numerical constant which depends on $\|\Rho\|_{C^{2,1}}$ and blows up when $\eta\to 0$ since we do not have estimates in $C^{2}$ of the solutions of the local Cahn-Hilliard equation. Now, we estimate the terms appearing on the (RHS) of \eqref{est:relative_entropy_3}.\\

\underline{\textit{Term $E$.}} For the nonlocal term $E$ we use boundedness of $\Bu$ to have 
$$
\left|\f{1}{\eps}\f{1}{\eta^{2}}\int_{0}^{t}\int_{\Td}({\rho}-\Rho) \Bu\cdot\nabla\omega_{\eta}\ast\left({\rho}-\Rho\right)\diff x\diff\tau\right| \leq \frac{C\, \|\Bu\|_{\infty}}{\eta^2}\, \|{\rho}-\Rho\|_2 \, \|\nabla\omega_{\eta}\ast\left({\rho}-\Rho\right) \|_{2} \leq \frac{C\|\Bu\|_{\infty}}{\eta^{d+3}} \|{\rho}-\Rho\|_2^2.
$$
Using \eqref{eq:rho-Rho_via_Ponce_bar} for $\eta \in (0,\eta_0)$ we obtain
$$
E \leq 
\f{C(\|\Rho\|_{C^{2,1}})}{4\eta^{d+3}}\int_0^t \Theta(\tau) \diff \tau.
$$

\underline{\textit{Term $B$.}} Using \eqref{eq:id_wl_1} and \eqref{eq:id_wl_3} we can write
\begin{align*}
B=-\int_{0}^{t}\int_{\Td}\left \langle \f{e(\Rho,\Bu)}{\Rho}\sqrt{\lambda_{1}}(\lambda'-\sqrt{\lambda}_{1} \Bu),\nu_{t,x}\right\rangle \diff x \diff \tau
\end{align*}
Using Cauchy-Schwartz with a parameter
$$
B \leq \int_{0}^{t}\int_{\Td}\left \langle \frac{\varepsilon^2}{2} \left|\f{e(\Rho,\Bu)}{\Rho}\right|^2\, \lambda_1 + \frac{1}{2\varepsilon^2}|\lambda'-\sqrt{\lambda}_{1} \Bu|^2,\nu_{t,x}\right\rangle  \diff x \diff \tau
$$
Now, $\left|\f{e(\Rho,\Bu)}{\Rho}\right| \leq \varepsilon\, C (\|\Rho\|_{C^{2,1}}) \, \left\| \frac{1}{\Rho}  \right\|_{\infty}$. Moreover, expanding the square in $|\lambda'-\sqrt{\lambda}_{1} \Bu|^2$ and using \eqref{eq:id_wl_3}, \eqref{eq:id_wl_5} we recognize that 
$$
\int_{0}^{t}\int_{\Td} \left \langle  |\lambda'-\sqrt{\lambda}_{1} \Bu|^2,\nu_{t,x}\right\rangle \leq  \int_{0}^{t}\int_{\Td} \overline{\rho|\bu-\Bu|^{2}}{ \, \diff x\diff t}.
$$
Therefore, we have the estimate
$$
B \leq \varepsilon^4 \, C(\|\Rho\|_{C^{2,1}}) \, \left\| \frac{1}{\Rho}  \right\|_{\infty}^2+\f{1}{2\eps^{2}}\int_{0}^{t}\int_{\Td}\overline{\rho|\bu-\Bu|^{2}}{\diff x\diff t}.
$$

\underline{\textit{Term $C$.}} We have
$$
\left|\f{1}{\eps}\int_{0}^{t}\int_{\Td}\nabla\Bu:\overline{\rho(\bu-\Bu)\otimes(\bu-\Bu)}\diff x\diff \tau \right| 
\leq \frac{C\,\|\nabla U\|_{\infty}}{\varepsilon} \int_{0}^{t}\int_{\Td}\left|\overline{\rho(\bu-\Bu)\otimes(\bu-\Bu)}\right| \diff x\diff \tau 
$$
Estimating directly under the integral in \eqref{eq:def_rel_kin_energy_tensor}
$$
\left|\langle (\lambda' - \sqrt{\lambda_1}\Bu) \otimes (\lambda' - \sqrt{\lambda_1}\Bu), \nu_{t,x}\rangle \right| \leq  \langle|\lambda' - \sqrt{\lambda_1}\Bu|^2, \nu_{t,x} \rangle
$$
and using \eqref{eq:inequality_2_conc_measures} we arrive at
$$
C \leq {C(\|P\|_{C^{2,1}})} \int_{0}^{t}\int_{\Td} \overline{\rho |\bu-\Bu|^2} \diff x\diff \tau
$$
\underline{\textit{Term $D$.}} Using \eqref{eq:def_rel_potenial_rho_Rho} and \eqref{eq:id_wl_10}, we can write
$$
|\overline{p({\rho}|\Rho)}| \leq \langle p(\lambda_1) - p(\Rho) - p'(\Rho) (\lambda_1 - \Rho), \nu_{t,x} \rangle + |m^{\rho\,F'(\rho)}| + m^{F(\rho)} + \frac{1}{\eta^2} \, m^{\rho^2}.
$$
The first part can be estimated using \eqref{eq:bound_rel_p_F}:
\begin{equation}\label{eq:estimate_p(rR)_part1}
\langle p(\lambda_1|\Rho), \nu_{t,x} \rangle \leq  C_{F,R} \langle F(\lambda_1|\Rho), \nu_{t,x} \rangle + 
\left(C_{F,R} + \frac{1}{\eta^2} \right) \langle (\lambda_1 - P)^2,\nu_{t,x} \rangle
\end{equation}
The concentration measures part can be estimated using \eqref{eq:inequality_2_conc_measures_func_p_F}:
\begin{equation}\label{eq:estimate_p(rR)_part2}
|m^{\rho\,F'(\rho)}| + m^{F(\rho)} + \frac{1}{\eta^2} \, m^{\rho^2} \leq ( C_F + 1) \, m^{F(\rho)} + \left(C_F + \frac{1}{\eta^2} \right)\, m^{\rho^2}.
\end{equation}
Summing up \eqref{eq:estimate_p(rR)_part1} and \eqref{eq:estimate_p(rR)_part2} we obtain
$$
|\overline{p({\rho}|\Rho)}| \leq C \, \overline{F(\rho|\Rho)} + C\,\left(1+\frac{1}{\eta^2}\right) \overline{|\rho - \Rho|^2}.
$$
The last term can be estimated by the nonlocal term appearing in the definition of $\Theta$ due to the Poincaré inequality \eqref{eq:nonloc_Poinc_stronger_bar}. As $\overline{F(\rho|\Rho)}$ also appears in the definition of $\Theta$ we obtain
$$
D \leq \left|\f{1}{\eps}\int_{0}^{t}\int_{\Td}\DIV(\Bu)\overline{p({\rho}|\Rho)}\diff x\diff \tau\right| \leq C(\|\Rho\|_{C^{2,1}}) \left(1+\f{1}{\eta^{2}}\right) \int_{0}^{t} \Theta(\tau) \diff \tau.
$$
We conclude that for $\eta < 1$:
$$
\Theta(t) \leq \Theta(0) + \f{C(\|\Rho\|_{C^{2,1}})}{4\eta^{d+3}}\int_0^t \Theta(\tau) \diff \tau + \varepsilon^4 \, C(\|\Rho\|_{C^{2,1}}) \, \left\| \frac{1}{\Rho}  \right\|_{\infty}^2
$$
Using Gronwall's lemma, we obtain \eqref{est:relative_entropy_4}.
\end{proof}

\begin{proof}[Proof of Theorem \ref{thm:conv_EK_CH_nonloc}] 
The proof is a direct consequence of \eqref{est:relative_entropy_4}. Indeed, we consider the relative entropy $\Theta$ as in \eqref{eq:relative_entropy} with $\rho = \rho_{\eta,\eps}$, $\bu = \bu_{\eta,\eps}$, $\Rho = \rho_{\eta}$ and $\Bu = -\eps\nabla(F'(\rho_{\eta})-B_{\eta}(\rho_{\eta}))$. As $\eta \in (0,\eta_0)$ is fixed, $\rho_{\eta}$ (which depends on $\eta$!) is a $C^{2,1}$ function bounded away from 0 (Theorem \ref{thm:class_sol_NLCH}, Lemma \ref{lem:class_sol_approx}). Furthermore, \begin{equation}\label{eq:initial_entropy}
\Theta(0) \leq C\,( \varepsilon^2 + \|{ \bu^0_{\varepsilon}}\|^2_{L^2(\Td)}) \to 0
\end{equation} 
(here, we use that the initial density ${\rho^0}$ belongs to $C^3$ so that $\|{\Bu(0,x)}\|_{L^{\infty}(\Td)} \leq C\, \varepsilon$, cf. \eqref{eq:defBu}). Therefore, we get that $\Theta(t) \to 0$ as $\varepsilon \to 0$. By \eqref{eq:nonnegativity_rel_entropy_explicit} and \eqref{eq:rho-Rho_via_Ponce_bar}, we obtain convergence in $L^2(\Td)$, even uniformly in time.
\end{proof}
\begin{proof}[Proof of Theorem \ref{thm:conv_EK_CH_loc}]
We write $\rho_{\eta}$ (note that it does not depend on $\varepsilon$, cf. \eqref{eq:EKCH1} and \eqref{eq:defBu}) for solutions to \eqref{eq:EKCH1}--\eqref{eq:EKCH2} and we note that they depend on $\eta$. From~\cite{MR4574535} we know that there exists a subsequence $\eta_{k} \to 0$ such that
$$
\|{\rho_{\eta_k}} - \rho \|_{L^2((0,T)\times \Td)} \to 0,
$$  
where $\rho$ is a weak solution to the local Cahn-Hilliard equation. Now, let $\rho_{\eta_k, \varepsilon_k}$ be a measure-valued solution of non-local Euler-Korteweg equation. Using \eqref{eq:initial_entropy}, {\eqref{eq:nonnegativity_rel_entropy_explicit} and \eqref{eq:rho-Rho_via_Ponce_bar}}, we have
$$
{\|\rho_{\eta_k} - \rho_{\eta_k, \varepsilon_k}\|_{L^2((0,T)\times\Td)} \leq C\left(\varepsilon^2_k + \|\bu_{0,\varepsilon_k}\|^2_{L^2(\Td)} + \eps_k^{4}\, \|\rho_{\eta_k}\|_{C^{2,1}}^2 \, \left\| \frac{1}{\rho_{\eta_k}}  \right\|_{\infty}^2 \right) \, e^{CT\|\rho_{\eta_k}\|_{C^{2,1}}/\eta^{d+3}}.}
$$
Of course, the quantity {$\|\rho_{\eta_k}\|_{C^{2,1}}^2 \, \left\| \frac{1}{\rho_{\eta_k}}  \right\|_{\infty}^2 \, e^{CT\|\rho_{\eta_k}\|_{C^{2,1}}/\eta_k^{d+3}}$} is blowing up as $\eta_k \to 0$ (because we lose parabolicity), nevertheless we can choose $\varepsilon_k$ so small to obtain convergence to 0. The conclusion follows by triangle inequality.
\end{proof}

\section{Convergence result for the parametrized measure $\nu^{\eta, \eps}$ and the concentration measures $m_{\eta, \eps}$}
Theorems \ref{thm:conv_EK_CH_loc} and \ref{thm:conv_EK_CH_nonloc} answer the question of what happens with the function $\rho_{\eta, \varepsilon}$ when $\eta, \varepsilon \to 0$. However, the measure-valued solution $(\rho_{\eta, \varepsilon}, \overline{\sqrt{\rho_{\eta, \varepsilon}}\bu_{\eta, \varepsilon}}, \nu^{\eta, \varepsilon},m_{\eta, \varepsilon})$ is in fact a collection of four components. Below, we address the question of convergence of the other components: $\overline{\sqrt{\rho_{\eta, \varepsilon}}\bu_{\eta, \varepsilon}}$, $\nu^{\eta, \varepsilon}$, $m_{\eta, \varepsilon}$. We provide a detailed proof only for the situation in Theorem \ref{thm:conv_EK_CH_nonloc}. Adaptation to the case analyzed in Theorem \ref{thm:conv_EK_CH_loc} is straightforward. \\

We first recall some basic notions from measure theory. We consider the set $\R^+\times \R^d$ and we write $(\lambda_1, \lambda')$ for a given element of this set where $\lambda_1 \in \R^+$ and $\lambda' \in \R^d$ as in Section \ref{sect:def_diss_mvs}. For two probability measures $\mu$, $\nu$ on $\R^+\times \R^d$ with a finite second moment, that is,
$$
\int_{\R^+\times \R^d} \left( |\lambda_1|^2 + |\lambda'|^2 \right) \diff \mu(\lambda_1, \lambda') < \infty, \quad \int_{\R^+\times \R^d} \left( |\lambda_1|^2 + |\lambda'|^2 \right) \diff \nu(\lambda_1, \lambda') < \infty,
$$
the Wasserstein distance $\mathcal{W}_2(\mu, \nu)$ is defined as
\begin{equation}\label{eq:def_W2_metric}
\mathcal{W}_2(\mu, \nu)^2 = \inf_{\pi \in \Pi(\mu, \nu)} \int_{(\R^+\times \R^d)^2} \left[ \left|\lambda_1 - \widetilde{\lambda_1}\right|^2 + \left|\lambda' - \widetilde{\lambda'}\right|^2  \right] \diff \pi\left(\lambda_1, \lambda', \widetilde{\lambda_1}, \widetilde{\lambda'}\right),
\end{equation}
where the set $\Pi(\mu, \nu)$ is the set of couplings between $\mu, \nu$; that is, the set of measures $\pi$ on the product $(\R^+\times \R^d)^2$ such that
$$
\pi(A \times (\R^+\times \R^d)) = \mu(A), \qquad \qquad \pi( (\R^+\times \R^d) \times B) = \nu(B).
$$
Furthermore, for a measure $\mu$ on some space $X$, the total variation of $\mu$ is defined as
$$
\|\mu\|_{TV} = |\mu|(X),
$$
where $|\mu|(A)= \mu^+(A)-\mu^-(A)$ and $\mu^+$, $\mu^-$ are positive and negative parts of $\mu$, respectively. Note that if $\mu$ is a nonnegative measure, $\|\mu\|_{TV} = \mu(X)$. For more on spaces of measures and related norms, we refer to \cite[Chapter 1]{MR4309603}.
\begin{thm}\label{thm:convergence_YM_only_eps}
Under the notation of Theorem \ref{thm:conv_EK_CH_nonloc}, the function $\overline{\sqrt{\rho_{\eta, \varepsilon}}\bu_{\eta, \varepsilon}}$ converges to 0 in $L^{\infty}(0,T;L^2(\Td))$:
\begin{equation}\label{eq:conv_another_parts_mvs_1}
\esssup_{t\in(0,T)} \int_{\Td} |\overline{\sqrt{\rho_{\eta, \varepsilon}}\bu_{\eta, \varepsilon}}|^2 \diff x \to 0 \mbox{ as } \eps \to 0.
\end{equation}
Moreover, the parametrized measure
$
\nu^{\eta,\eps} \in L^{\infty}_{\text{weak}}((0,T)\times\Td;\mathcal{P}([0,+\infty)\times\mathbb{R}^{d}))
$
converges to $\delta_{\rho_{\eta}(t,x)} \otimes \delta_{\bm{0}}$ in the following sense
\begin{equation}\label{eq:conv_another_parts_mvs_2}
\esssup_{t\in(0,T)} \int_{\Td} \left[\mathcal{W}_2(\nu^{\eta, \eps}, \delta_{\rho_{\eta}(t,x)}\otimes \delta_{\bm{0}})\right]^2\diff x \to 0 \mbox{ as } \varepsilon \to 0.
\end{equation}
Furthermore, the concentration measures vector $m_{\eta,\eps}$ converges to 0 in the total variation norm, uniformly in time:
\begin{equation}\label{eq:conv_another_parts_mvs_3}
\esssup_{t\in(0,T)} \|m_{\eta,\eps}(t,\cdot)\|_{TV} \to 0 \mbox{ as } \varepsilon \to 0.
\end{equation}
\end{thm}
\begin{proof}
From the proof of Theorem \ref{thm:conv_EK_CH_nonloc}, we know that $\sup_{t\in(0,T)} \Theta(t) \to 0$ where $\Theta(t)$ is defined as in \eqref{eq:relative_entropy} with $\rho := \rho_{\eta,\eps}$, $\Rho := \rho_{\eta}$, $\bu := \bu_{\eta,\eps}$ and
\begin{equation}\label{eq:form_of_Bu_with_rho_eta}
\Bu := -\eps\nabla(F'(\rho_\eta)-B_{\eta}(\rho_\eta)).
\end{equation}
Due to Lemma \ref{lem:rel_entr_nonneg} this yields
\begin{equation}\label{eq:convergence_2_terms_in_rel_entropy}
\sup_{t\in(0,T)} \int_{\Td}\f{1}{2}\overline{\rho_{\eta,\eps}|\bu_{\eta,\eps}-\Bu|^{2}} \diff x +\f{\kappa}{4\eta^{2}}\int_{\Td}\overline{\int_{\Td}\omega_{\eta}(y)|(\rho_{\eta,\eps}-\rho_{\eta})(x)-(\rho_{\eta,\eps}-\rho_{\eta})(x-y)|^{2}\diff y} \diff x \to 0    
\end{equation}
and these two quantities are nonnegative. First, by Poincaré inequality \eqref{eq:nonloc_Poinc_stronger_bar} and \eqref{eq:def_square_diff_densities}, we have
\begin{equation}\label{eq:conv_W2_first_part_with_conc_measure}
\esssup_{t\in(0,T)} \int_{\Td} \int_{\R^+ \times \R^d} |\lambda_1 - \rho_{\eta}(t,x)|^2 \diff \nu^{\eta,\eps}_{t,x}(\lambda_1,\lambda')\diff x + m^{\rho^2}_{\eta,\eps}(t,\Td)  \to 0.
\end{equation}
In particular,
\begin{equation}\label{eq:conv_W2_first_part}
\esssup_{t\in(0,T)} \int_{\Td} \int_{\R^+} \int_{\R^+ \times \R^d} |\lambda_1 - \widetilde{\lambda_1}(t,x)|^2 \diff \nu^{\eta,\eps}_{t,x}(\lambda_1,\lambda') \diff \delta_{\rho_{\eta}(t,x)}(\widetilde{\lambda_1}) \diff x \to 0.
\end{equation}
Second, due to \eqref{eq:def_rel_kin_energy}, we can expand the term $\int_{\Td}\f{1}{2}\overline{\rho_{\eta,\eps}|\bu_{\eta,\eps}-\Bu|^{2}}$ into three integrals:
\begin{equation}\label{eq:split_rel_kin_energy_three_terms}
\frac{1}{2} \int_{\Td} \overline{\rho_{\eta,\eps}\,|\bu_{\eta,\eps} |^2} \diff x - \int_{\Td} \overline{\rho_{\eta,\eps}\, \bu_{\eta, \eps}} \cdot \Bu \diff x + \frac{1}{2} \int_{\Td} {\rho_{\eta,\eps}}\, |\Bu|^2 \diff x.
\end{equation}
We claim that the second and third term converge to 0. For the third term, this follows easily from \eqref{eq:form_of_Bu_with_rho_eta}, nonnegativity of $\rho_{\eta,\eps}$ and the conservation of mass $\int_{\Td} \rho_{\eta,\eps} \diff x = \int_{\Td} {\rho^0} \diff x$. Concerning the second term, by the dissipativity (Definition \ref{def:diss_mvs}) and nonnegativity of the energy (Lemma \ref{lem:nonneg_the_energy_one}), we have the uniform estimate
\begin{multline*}
\left| \int_{\Td} \overline{\rho_{\eta,\eps}\, \bu_{\eta, \eps}} \diff x \right| = \left| \int_{\Td} \left\langle \sqrt{\lambda_1}\, \lambda' , \nu^{\eta,\eps} \right\rangle \diff x \right| \leq \frac{1}{2} \int_{\Td} \langle \lambda_1, \nu^{\eta,\eps} \rangle \diff x+ \frac{1}{2} \int_{\Td} \langle |\lambda'|^2, \nu^{\eta,\eps} \rangle \diff x \leq \\ \leq \frac{1}{2} \int_{\Td} \rho_{\eta,\eps} \diff x + \frac{1}{2} \int_{\Td} \overline{|\sqrt{\rho_{\eta,\eps}}\bu_{\eta,\eps}|^2} \diff x = \frac{1}{2} \int_{\Td} {\rho^0} \diff x + \frac{1}{2} \int_{\Td} \overline{|\sqrt{\rho_{\eta,\eps}}\bu_{\eta,\eps}|^2} \diff x \leq C
\end{multline*}
As $|\Bu|\leq C\varepsilon$, we conclude that $\esssup_{t\in(0,T)} \left| \int_{\Td} \overline{\rho_{\eta,\eps}\, \bu_{\eta, \eps}} \cdot \Bu \diff x \right| \to 0$ as $\varepsilon \to 0$ so that \eqref{eq:split_rel_kin_energy_three_terms} implies
$$
\esssup_{t\in(0,T)} \frac{1}{2} \int_{\Td} \overline{\rho_{\eta,\eps}\,|\bu_{\eta,\eps} |^2} \diff x  \to 0.
$$
Again, we can write it as
\begin{equation}\label{eq:conv_W2_second_part_with_conc_measure}
\esssup_{t\in(0,T)} \int_{\Td} \int_{\R^+ \times \R^d} |\lambda'|^2 \diff \nu^{\eta,\eps}_{t,x}(\lambda_1,\lambda') \diff x + m^{\rho\,|\bu|^2}_{\eta,\eps}(t,\Td) \to 0
\end{equation}
which implies 
\begin{equation}\label{eq:conv_W2_second_part}
\esssup_{t\in(0,T)} \int_{\Td} \int_{\R^d} \int_{\R^+ \times \R^d} |\lambda' - \widetilde{\lambda'}|^2 \diff \nu^{\eta,\eps}_{t,x}(\lambda_1,\lambda') \diff \delta_{\bm{0}}(\widetilde{\lambda'}) \diff x  \to 0.
\end{equation}
Now, as the product measure $\nu^{\eta,\eps}_{t,x}(\lambda_1, \lambda') \otimes \delta_{\rho_{\eta}}(\widetilde{\lambda_1}) \otimes \delta_{\bm{0}}(\widetilde{\lambda'})$ is an admissible coupling between $\nu^{\eta,\eps}_{t,x}$ and $\delta_{\rho_{\eta}} \otimes \delta_{\bm{0}}$ we can estimate the infimum in \eqref{eq:def_W2_metric} by
\begin{multline*}
\left[\mathcal{W}_2\left(\nu^{\eta,\eps}_{t,x},  \delta_{\rho_{\eta}} \otimes \delta_{\bm{0}}  \right) \right]^2 \leq
\int_{\R^+} \int_{\R^+ \times \R^d} |\lambda_1 - \widetilde{\lambda_1}(t,x)|^2 \diff \nu^{\eta,\eps}_{t,x}(\lambda_1,\lambda') \diff \delta_{\rho_{\eta}(t,x)}(\widetilde{\lambda_1})
\,+  \\ +
\int_{\R^d} \int_{\R^+ \times \R^d} |\lambda' - \widetilde{\lambda'}|^2 \diff \nu^{\eta,\eps}_{t,x}(\lambda_1,\lambda') \diff \delta_{\bm{0}}(\widetilde{\lambda'}) 
\end{multline*}
so that integrating over $\Td$ and taking $\esssup_{t\in(0,T)}$ we conclude the proof of \eqref{eq:conv_another_parts_mvs_2} due to \eqref{eq:conv_W2_first_part} and \eqref{eq:conv_W2_second_part}. Furthermore, by Jensen's inequality
$$
\int_{\Td} |\overline{\sqrt{\rho_{\eta, \varepsilon}}\bu_{\eta, \varepsilon}}|^2 {\diff x}= \int_{\Td} |\langle \lambda', \nu^{\eta,\eps} \rangle|^2 \diff x \leq
\int_{\Td} \langle |\lambda'|^2, \nu^{\eta,\eps} \rangle { \diff x}.
$$
Taking $\esssup_{t\in(0,T)}$ and using \eqref{eq:conv_W2_second_part_with_conc_measure}, we arrive at \eqref{eq:conv_another_parts_mvs_1}.\\

Finally, we study the concentration measures. From \eqref{eq:conv_W2_first_part_with_conc_measure} and \eqref{eq:conv_W2_second_part_with_conc_measure} we know that 
$$
\esssup_{t\in(0,T)} m^{\rho^2}_{\eta,\eps}(t,\Td),\quad \esssup_{t\in(0,T)} m^{\rho\,|\bu|^2}_{\eta,\eps}(t,\Td) \to 0 \mbox{ as } \eps \to 0.
$$
Using \eqref{eq:inequality_2_conc_measures} we obtain the same for $|m^{\rho \, \bu\otimes \bu}_{\eta,\eps}|$. It remains to study $m^{F(\rho)}_{\eta, \eps}$ and $m^{\rho\,F'(\rho)}_{\eta, \eps}$. In fact, if we prove that $\esssup_{t\in(0,T)} m^{F(\rho)}_{\eta, \eps}(t,\Td)$ converges to $0$ as $\eps \to 0$, the same will be true for $\left|m^{\rho\,F'(\rho)}_{\eta, \eps}\right|$ due to \eqref{eq:inequality_2_conc_measures_func_p_F}.\\

By $\sup_{t\in(0,T)} \Theta(t) \to 0$ and \eqref{eq:convergence_2_terms_in_rel_entropy}, we have that
$$
\sup_{t\in(0,T)} \int_{\Td} \overline{F(\rho_{\eta,\eps}|\rho_{\eta})} \, {  \diff x} \to 0 \mbox{ as } \eps \to 0.
$$
We can write $\overline{F(\rho_{\eta,\eps}|\rho_{\eta})}$ as (cf. \eqref{eq:F_rel_pot_splitted_for_two_parts_mvs}) 
\begin{equation}\label{eq:F_rho_eta_eta_eps_splitting}
\begin{split}
\overline{F(\rho_{\eta,\eps}|\rho_{\eta})} = &\left \langle F_1(\lambda_1) - F_1(\rho_{\eta}) - F_1'(\rho_{\eta}) (\lambda_1 - \rho_{\eta}), \nu^{\eta,\eps} \right \rangle
+\\ &\qquad \qquad+ \left \langle F_2(\lambda_1) - F_2(\rho_{\eta}) - F_2'(\rho_{\eta}) (\lambda_1 - \rho_{\eta}), \nu^{\eta,\eps} \right \rangle  + m^{F(\rho)}_{\eta, \eps}.
\end{split}
\end{equation}
The first term is nonnegative while the second converges to 0. Indeed, it can be bounded by $\|F_2''\|_{\infty} \left \langle (\lambda_1 - \rho_{\eta})^2, \nu_{t,x} \right \rangle$ which can be estimated due to inequality (cf. \eqref{eq:F_rel_pot_splitted_for_two_parts_mvs_2}):
$$
 \|F_2''\|_{\infty} \int_{\Td} \left \langle (\lambda_1 - \rho_{\eta})^2, \nu_{t,x} \right \rangle \diff x  \leq 
\frac{1-\kappa}{4\eta^2}\int_{\Td}\overline{\int_{\Td}\omega_{\eta}(y)|(\rho-\rho_{\eta})(x)-(\rho-\rho_{\eta})(x-y)|^{2}\diff y} \diff x
$$
for some $\kappa \in (0,1)$. Thanks to \eqref{eq:convergence_2_terms_in_rel_entropy},
$$ 
\esssup_{t\in(0,T)} \int_{\Td} \left|\left \langle F_2(\lambda_1) - F_2(\rho_{\eta}) - F_2'(\rho_{\eta}) (\lambda_1 - \rho_{\eta}), \nu^{\eta,\eps} \right \rangle\right| \diff x \to 0 \mbox{ as } \eps \to 0.
$$
Due to \eqref{eq:F_rho_eta_eta_eps_splitting}, the proof of \eqref{eq:conv_another_parts_mvs_3} is concluded.
\end{proof}

We can also formulate a similar result to Theorem \ref{thm:convergence_YM_only_eps} in the context of Theorem \ref{thm:conv_EK_CH_loc}. The proof is the same as the one of Theorem \ref{thm:convergence_YM_only_eps}.
\begin{thm}\label{thm:convergence_YM_eps_and_eta}
Under the notation of Theorem \ref{thm:conv_EK_CH_loc}, the function $\overline{\sqrt{\rho_{\eta_k, \varepsilon_k}}\bu_{\eta_k, \varepsilon_k}}$ converges to 0 in $L^{\infty}(0,T;L^2(\Td))$:
\begin{equation*}
\esssup_{t\in(0,T)}  \int_{\Td} |\overline{\sqrt{\rho_{\eta_k, \varepsilon_k}}\bu_{\eta_k, \varepsilon_k}}|^2 \diff x  \to 0 \mbox{ as } \eps_k, \eta_k \to 0.
\end{equation*}
Moreover, the parametrized measure
$
\nu^{\eta_k,\eps_k} \in L^{\infty}_{\text{weak}}((0,T)\times\Td;\mathcal{P}([0,+\infty)\times\mathbb{R}^{d}))
$
converges to $\delta_{\rho(t,x)} \otimes \delta_{\bm{0}}$ in the following sense:
\begin{equation*} 
\int_0^T \int_{\Td} \left[\mathcal{W}_2(\nu^{\eta_k, \eps_k}, \delta_{\rho(t,x)}\otimes \delta_{\bm{0}})\right]^2\diff x \diff t \to 0 \mbox{ as } \varepsilon_k, \eta_k \to 0.
\end{equation*}
Furthermore, the concentration measures vector $m_{\eta_k,\eps_k}$ converges to 0 in the total variation norm, uniformly in time:
\begin{equation*}
\esssup_{t\in(0,T)} \|m_{\eta_k,\eps_k}(t,\cdot)\|_{TV} \to 0 \mbox{ as } \varepsilon_k, \eta_k \to 0.
\end{equation*}
\end{thm}
\appendix

\section{Some inequalities}

\begin{lem}
Let $\sigma>0$ and a final time $T>0$. Let $\underline{u}$ be defined by $\underline{u}(t)=\sigma\exp\left(-\int_{0}^{t}\norm{\DIV b}_{L^{\infty}}(s)\diff s\right)$ and $\phi_{\delta}$ defined in~\eqref{eq:phi_delta}. Then
\begin{equation*}
\int_{\Td}\Delta\phi_{\delta}(u)\sgn^{-}(u-\underline{u}){\diff x}\ge 0.     
\end{equation*}
\end{lem}
\begin{proof}
We note $f_{\tau}$ a concave approximation as $\tau\to 0$ of the function $f:x\mapsto\min\{x,0\}$. Then $f_{\tau}'$ approximates $f':x\mapsto\sgn^{-}(x)$. 
We have
\begin{align*}
\int_{\Td}\Delta\phi_{\delta}(u)f_{\tau}'(u-\underline{u}){\diff x}&=-\int_{\Td}\phi_{\delta}'(u)f_{\tau}''(u-\underline{u})|\nabla u|^{2}{\diff x}.   
\end{align*}
Since $\phi_{\delta}'\ge 0$, $f_{\tau}''\le 0$ by concavity and we conclude by sending $\tau\to 0$. 
\end{proof}

\begin{lem}\label{lem:Poincaré_nonlocal_H1_L2}
There exists ${\eta_0}>0$ and constant $C_P$ such that for all $\eta \in (0, \eta_0)$ and all $f \in L^2(\Td)$ we have
$$
\| f - (f)_{\Td} \|^2_{L^2(\Td)} \leq  C_P\, \int_{\Td} \int_{\Td}  \frac{| f(x) - f(y)|^2}{4\eta^2} \omega_\eta(|x-y|) \diff x \diff y .
$$
\end{lem}

\section{Bound on the relative pressure}\label{app:relative_pressure}
\begin{lem}
Let $F$ satisfy Assumption~\eqref{ass:potentialF}, $p(\rho) = \rho F'(\rho)-F(\rho)+\f{\rho^{2}}{2\eta^{2}}$ and $F(\rho|\Rho)$, $p(\rho|\Rho)$ be defined by \eqref{eq:F(rho-Rho)}. Then there exists a constant $C_{F,R}$ such that $p(\rho|\Rho)$ is bounded in terms of $F(\rho|\Rho)$ and $|{\rho}-\Rho|^{2}$ \ie 
\begin{equation}\label{eq:bound_rel_p_F}
p({\rho}|\Rho)\le C_{F,R} \, F({\rho}|\Rho) + \left(C_{F,R} + \f{1}{\eta^{2}}\right) |{\rho}-\Rho|^{2}.
\end{equation}
Similarly, there exists constant $C_F$ such that
\begin{equation}\label{eq:bound_notrel_p_F}
|\rho\,F'(\rho)| \leq C_F \, F(\rho) + C_F\, \rho^2 + C_F.
\end{equation}
\end{lem}
\begin{proof}
We write 
\begin{equation*}
p({\rho}|\Rho)=({\rho}-\Rho)^{2}\int_{0}^{1}\int_{0}^{\tau}p''(s{\rho}+(1-s)\Rho)\diff s\diff\tau,    
\end{equation*}
\begin{equation*}
F({\rho}|\Rho)=({\rho}-\Rho)^{2}\int_{0}^{1}\int_{0}^{\tau}F''(s{\rho}+(1-s)\Rho)\diff s\diff\tau.    
\end{equation*}
We note $h(s)=s {\rho}+(1-s)\Rho$ to simplify the notations. By definition $p'(\rho)=\rho\,(F''(\rho)+\f{1}{\eta^{2}})$. Therefore we obtain
\begin{align*}
p({\rho}|\Rho)&=({\rho}-\Rho)^{2}\int_{0}^{1}\int_{0}^{\tau}F_{1}''(h(s))+F_{2}''(h(s))+h(s)F_{1}^{(3)}(h(s))+h(s)F_{2}^{(3)}(h(s))\diff s\diff\tau+\f{1}{\eta^{2}}|{\rho}-\Rho|^{2}\\
&=F({\rho}|\Rho)+({\rho}-\Rho)^{2}\int_{0}^{1}\int_{0}^{\tau}h(s)F_{1}^{(3)}(h(s))+h(s)F_{2}^{(3)}(h(s))\diff s\diff\tau+\f{1}{\eta^{2}}|{\rho}-\Rho|^{2}.
\end{align*}

We note $I_{1}=\int_{0}^{1}\int_{0}^{\tau}h(s)F_{1}^{(3)}(h(s))\diff s\diff\tau$ and $I_{2}=\int_{0}^{1}\int_{0}^{\tau}h(s)F_{2}^{(3)}(h(s))\diff s\diff\tau$.
By assumptions on $|uF^{(3)}_{1}|$ we obtain 
\begin{equation*}
I_{1}\le C+C\int_{0}^{1}\int_{0}^{\tau}F_{1}''(h(s))\diff s\diff\tau\le C+C\int_{0}^{1}\int_{0}^{\tau}F_{1}''(h(s))+F_{2}''(h(s))\diff s \diff\tau,     
\end{equation*}
where the value of $C$ changed in the last inequality, using the boundedness assumption on $F_{2}''$. For $I_{2}$ we simply use boundedness of $|u F^{3}_2(u)|$ so that
\begin{equation*}
I_{2}\le C.
\end{equation*}
This concludes the proof of \eqref{eq:bound_rel_p_F}. Concerning \eqref{eq:bound_notrel_p_F}, we have
$$
\rho\, F'(\rho) = \rho\, F_1'(\rho) + \rho\, F_2'(\rho) \leq
C(1+ F_1(\rho)) + C \, \rho \leq C\,(1+F(\rho)) + C\left(\frac{1}{2} + \frac{\rho^2}{2}\right) + C,
$$
where we used estimate on $\rho\, F_1'(\rho)$, boundedness of $F_2'$, $F_2$ and inequality $2\rho \leq 1+\rho^2$. The proof is concluded.
\end{proof}

\bibliographystyle{abbrv}
\bibliography{fastlimit}
\end{document}